\documentclass[smallextended,envcountsect]{svjour3}
\usepackage{amssymb,amsmath,vmargin}
\newcommand{\ind}{1\hspace{-.27em}\mbox{\rm l}}
\newcommand{\lqn}[1]{\noalign{\noindent $\displaystyle{#1}$}}\newcommand{\dd}{\mathrm{d}}
\numberwithin{equation}{section}

\begin{document}
\title{\centerline{A class of bridges of iterated integrals of Brownian motion}
\centerline{related to various boundary value problems involving}
\centerline{the one-dimensional polyharmonic operator}}
\titlerunning{Bridges of iterated integrals of Brownian motion related to various
boundary value problems}
\author{\centerline{Aim\'e LACHAL}}
\institute{
\textsc{Institut National des Sciences Appliqu\'ees de Lyon}\\
P\^ole de Math\'ematiques/Institut Camille Jordan CNRS UMR5208\\
B\^atiment L\'eonard de Vinci, 20 avenue Albert Einstein\\
69621 Villeurbanne Cedex, \textsc{France}\\
\email{aime.lachal@insa-lyon.fr}\\
Web page: http://maths.insa-lyon.fr/$\mbox{}^{\sim}$lachal
}
\maketitle
\begin{abstract}
Let $(B(t))_{t\in [0,1]}$ be the linear Brownian motion and $(X_n(t))_{t\in [0,1]}$
be the $(n-1)$-fold integral of Brownian motion, $n$ being a positive integer:
$$
X_n(t)=\int_0^t \frac{(t-s)^{n-1}}{(n-1)!} \,\dd B(s)\quad\mbox{for any $t\in[0,1]$.}
$$
In this paper we construct several bridges between times $0$ and $1$ of the process
$(X_n(t))_{t\in [0,1]}$ involving conditions on the successive derivatives
of $X_n$ at times $0$ and $1$. For this family of bridges, we make
a correspondance with certain boundary value problems related to
the one-dimensional polyharmonic operator. We also study the classical problem
of prediction. Our results involve various Hermite interpolation polynomials.
\end{abstract}

\section{Introduction}

Throughout the paper, we shall denote, for any enough differentiable function $f$,
its $i$-th derivative by $f^{(i)}$ or $\dd^i f/\dd t^i$.

Let $(B(t))_{t\in [0,1]}$ be the linear Brownian motion started at $0$ and
$(\beta(t))_{t\in [0,1]}$ be the linear Brownian bridge within the time interval
$[0,1]$: $(\beta(t))_{t\in [0,1]}=(B(t)|B(0)=B(1)=0)_{t\in [0,1]}$.
These processes are Gaussian processes with covariance functions
$$
c_{_B}(s,t)=s\wedge t\quad\mbox{and}\quad c_{_{\beta}}(s,t)=s\wedge t-st.
$$
For a given continuous function $u$, the functions $v_{_B}$ and $v_{_{\beta}}$
respectively defined on $[0,1]$ by
$$
v_{_B}(t)=\int_0^1 c_{_B}(s,t)u(s)\,\dd s\quad\mbox{and}\quad
v_{_{\beta}}(t)=\int_0^1 c_{_{\beta}}(s,t)u(s)\,\dd s
$$
are the solutions of the respective boundary value problems on $[0,1]$:
$$
\begin{cases}
v_{_B}''=-u, \\
v_{_B}(0)=v_{_B}'(1)=0,
\end{cases}
\quad\mbox{and}\quad
\begin{cases}
v_{_{\beta}}''=-u, \\
v_{_{\beta}}(0)=v_{_{\beta}}(1)=0.
\end{cases}
$$
Observe that the differential equations are the same in both cases.
Only the boundary conditions differ. They are Dirichlet-type boundary conditions
for Brownian bridge while they are Dirichlet/Neumann-type boundary conditions
for Brownian motion.

These well-known connections can be extended to the polyharmonic operator
$d^{2n}/dt^{2n}$ where $n$ is a positive integer.
This latter is associated with the $(n-1)$-fold integral
of Brownian motion $(X_n(t))_{t\in[0,1]}$:
$$
X_n(t)=\int_0^t \frac{(t-s)^{n-1}}{(n-1)!} \,\dd B(s)
\quad\mbox{for any $t\in[0,1]$.}
$$
(Notice that all of the derivatives at time~$0$ naturally vanish:
$X_n(0)=X_{n-1}(0)=\dots=X_2(0)=X_1(0)=0$.)
Indeed, the following facts for instance are known (see, e.g., \cite{bahadur}
and \cite{bridge}):
\begin{itemize}
\item
The covariance fonction of the process $(X_n(t))_{t\in[0,1]}$ coincide with
the Green function of the boundary value problem
$$
\begin{cases}
v^{(2n)}=(-1)^nu\quad\mbox{on }[0,1], \\
v(0)=v'(0)=\dots=v^{(n-1)}(0)=0, \\
v^{(n)}(1)=v^{(n+1)}(1)=\dots=v^{(2n-1)}(1)=0;
\end{cases}
$$
\item
The covariance fonction of the bridge $(X_n(t)|X_n(1)=0)_{t\in[0,1]}$
coincide with the Green function of the boundary value problem
$$
\begin{cases}
v^{(2n)}=(-1)^nu\quad\mbox{on }[0,1], \\
v(0)=v'(0)=\dots=v^{(n-1)}(0)=0, \\
v^{(n-1)}(1)=v^{(n+1)}(1)=\dots=v^{(2n-1)}(1)=0;
\end{cases}
$$
\item
The covariance fonction of the bridge $(X_n(t)|X_n(1)=X_{n-1}(1)
=\dots=X_1(1)=0)_{t\in[0,1]}$
coincide with the Green function of the boundary value problem
$$
\begin{cases}
v^{(2n)}=(-1)^nu\quad\mbox{on }[0,1], \\
v(0)=v'(0)=\dots=v^{(n-1)}(0)=0, \\
v(1)=v'(1)=\dots=v^{(n-1)}(1)=0.
\end{cases}
$$
\end{itemize}
Observe that the differential equations and the boundary conditions
at~$0$ are the same in all cases. Only the boundary conditions at~$1$ differ.
Other boundary value problems can be found in~\cite{hn1} and~\cite{hn2}.

We refer the reader to~\cite{dolph} for a pioneering work dealing with the
connections between general Gaussian processes and Green functions; see
also~\cite{carraro}. We also refer
to~\cite{chen},~\cite{ptreg},~\cite{lli},~\cite{lin1},~\cite{nik},~\cite{lin2} and the
references therein for various properties, namely asymptotical study,
of the iterated integrals of Brownian motion as well as
to~\cite{hn1},~\cite{hn2},~\cite{wah1} and~\cite{wah2} for interesting
applications of these processes to statistics.

The aim of this work is to examine all the possible conditioned processes
of $(X_n(t))_{t\in[0,1]}$ involving different events at time~$1$:
$$
(X_n(t)|X_{j_1}(1)=X_{j_2}(1)=\dots=X_{j_m}(1)=0)_{t\in[0,1]}
$$
for a certain number $m$ of events, $1\le m\le n$, and certain indices
$j_1,j_2,\dots,j_m$ such that $1\le j_1<j_2<\dots<j_m\le n$,
and to make the connection with the boundary value problems:
\begin{align*}
\begin{cases}
v^{(2n)}=(-1)^nu\quad\mbox{on }[0,1], \\
v(0)=v'(0)=\dots=v^{(n-1)}(0)=0, \\
v^{(i_1)}(1)=v^{(i_2)}(1)=\dots=v^{(i_n)}(1)=0
\end{cases}
\\[-7ex]
\noalign{\hfill(BVP)}
\end{align*}
for certain indices $i_1,i_2,\dots,i_n$ such that $0\le i_1<i_2<\dots<i_n\le 2n-1$.
Actually, we shall see that this connection does not recover all the
possible boundary value problems and we shall characterize those
sets of indices for which such a connection exists.

The paper is organized as follows.
In Section~\ref{sect-gaussian}, we exhibit the relationships between
general Gaussian processes and Green functions of certain boundary value
problems. In Section~\ref{sect-iteratedIBM}, we consider the iterated integrals
of Brownian motion. In Section~\ref{sect-bridges}, we construct several bridges
associated with the foregoing processes and depict explicitly their connections
with the polyharmonic operator together with various boundary conditions.
One of the main results is Theorem~\ref{th-BVP-gene}.
Moreover, we exhibit several interesting properties of the bridges
(Theorems~\ref{th-drift} and \ref{th-decompo-cov}) and solve the prediction
problem (Theorems~\ref{th-prediction}). In Section~\ref{sect-IBM}, we illustrate
the previous results on the case $n=2$ related to integrated Brownian motion.
Finally, in Section~\ref{sect-general}, we give a characterization for
the Green function of the boundary value problem (BVP) to be a covariance function.
Another one of the main results is Theorem~\ref{Green-sym}.

\section{Gaussian processes and Green functions}\label{sect-gaussian}

We consider a $n$-Markov Gaussian process $(X(t))_{t\in[0,1]}$ evolving on the
real line $\mathbb{R}$. By ``$n$-Markov'', it is understood that the trajectory
$t\mapsto X(t)$ is $n$ times differentiable and the $n$-dimensional process
$(X(t),X'(t),\dots,X^{(n-1)}(t))_{t\in[0,1]}$
is a Markov process. Let us introduce the covariance function of
$(X(t))_{t\in[0,1]}$: for $s,t\in[0,1]$, $c_{_X}(s,t)=\mathbb{E}[X(s)X(t)]$.
It is known (see~\cite{carraro}) that the function $c_{_X}$ admits the
following representation:
\begin{equation}\label{cov-gene}
c_{_X}(s,t)=\sum_{k=0}^{n-1} \varphi_k(s\wedge t) \psi_k(s\vee t)
\end{equation}
where $\varphi_k,\psi_k$, $k\in\{0,1,\dots,n-1\}$, are certain functions.

Let $\mathcal{D}_0,\mathcal{D}_1$ be linear differential operators of order less
than $p$ and let $\mathcal{D}$ be a linear differential operator of order $p$
defined by
$$
\mathcal{D}=\sum_{i=0}^{p} \alpha_i\frac{\dd^i}{\dd t^i}
$$
where $\alpha_0,\alpha_1,\dots,\alpha_p$ are continuous functions on $[0,1]$.
More precisely, we have for any $p$ times differentiable function $f$ and
any $t\in[0,1]$,
$$
(\mathcal{D} f)(t)=\sum_{i=0}^{p} \alpha_i(t) f^{(i)}(t).
$$
%
\begin{theorem}\label{th-gene}
Assume that the functions $\varphi_k,\psi_k$, $k\in\{0,1,\dots,n-1\}$, are
$p$~times differentiable and satisfy the following conditions, for a certain
constant $\kappa$:
\begin{equation}\label{condition-sum}
\sum_{k=0}^{n-1}\left[\varphi_k\psi_k^{(i)}-\varphi_k^{(i)}\psi_k\right]=
\begin{cases} 0 &\mbox{if }\, 0\le i\le p-2,\\ \kappa & \mbox{if }\, i=p-1,
\end{cases}
\end{equation}
\begin{equation}\label{condition-deriv}
\mathcal{D}\varphi_k=\mathcal{D}\psi_k=0,\quad(\mathcal{D}_0 \varphi_k)(0)=0,
\quad (\mathcal{D}_1 \psi_k)(1)=0.
\end{equation}
Then, for any continuous function $u$ on $[0,1]$, the function $v$ defined
on $[0,1]$ by
$$
v(t)=\int_0^1 c_{_X}(s,t)u(s)\,\dd s
$$
solves the boundary value problem
\begin{equation}\label{BVP-gene}
\begin{cases}
\mathcal{D} v= \alpha_p u, \\ (\mathcal{D}_0 v)(0)=(\mathcal{D}_1 v)(1)=0.
\end{cases}
\end{equation}
\end{theorem}
%
\begin{remark}
If the problem~(\ref{BVP-gene}) is determining, that is if it has a unique solution,
then the covariance function $c_{_X}$ is exactly the Green function of the
boundary value problem~(\ref{BVP-gene}).
\end{remark}
%
\begin{proof}
In view of~(\ref{cov-gene}), the function $v$ can be written as
$$
v(t)=\sum_{k=0}^{n-1}\left[\psi_k(t)
\int_0^t \varphi_k(s)u(s)\,\dd s+\varphi_k(t) \int_t^1 \psi_k(s)u(s)\,\dd s\right]\!.
$$
The derivative of $v$ is given by
\begin{align*}
v'(t)
&
=\sum_{k=0}^{n-1}\left[\psi'_k(t) \int_0^t \varphi_k(s)u(s)\,\dd s
+\varphi_k'(t) \int_t^1 \psi_k(s)u(s)\,\dd s\right]
\end{align*}
and its second order derivative, since
$\sum_{k=0}^{n-1}\left[\varphi_k\psi_k'-\varphi_k'\psi_k\right]=0$, by
\begin{align*}
v''(t)
&
=\sum_{k=0}^{n-1}\left[\psi''_k(t) \int_0^t \varphi_k(s)u(s)\,\dd s
+\varphi''_k(t) \int_t^1 \psi_k(s)u(s)\,dd s+\left[\varphi_k(t)\psi'_k(t)
-\varphi'_k(t)\psi_k(t)\right]u(t)\right]
\\
&
=\sum_{k=0}^{n-1}\left[\psi''_k(t) \int_0^t \varphi_k(s)u(s)\,\dd s
+\varphi''_k(t) \int_t^1 \psi_k(s)u(s)\,\dd s\right]\!.
\end{align*}
More generally, because of the assumptions~(\ref{condition-sum}),
we easily see that, for $i\in\{0,1,\dots, p-1\}$,
\begin{align*}
v^{(i)}(t)
&
=\sum_{k=0}^{n-1}\left[\psi^{(i)}_k(t) \int_0^t \varphi_k(s)u(s)\,\dd s
+\varphi^{(i)}_k(t) \int_t^1 \psi_k(s)u(s)\,\dd s\right]
\end{align*}
and the $p$-th order derivative of $v$ is given by
\begin{align*}
v^{(p)}(t)
&
=\sum_{k=0}^{n-1}\left[\psi^{(p)}_k(t) \int_0^t \varphi_k(s)u(s)\,\dd s
+\varphi^{(p)}_k(t) \int_t^1 \psi_k(s)u(s)\,\dd s
\right.
\\
&\hphantom{=\,}
\left.+\left[\varphi_k(t)\psi^{(p-1)}_k(t)-\varphi^{(p-1)}_k(t)\psi(t)\right]u(t)\right]
\\
&
=\sum_{k=0}^{n-1}\left[\psi^{(p)}_k(t) \int_0^t \varphi_k(s)u(s)\,\dd s
+\varphi^{(p)}_k(t) \int_t^1 \psi_k(s)u(s)\,\dd s\right]+\kappa u(t).
\end{align*}
Actually, we have proved that, for $i\in\{0,1,\dots, p-1\}$,
\begin{equation}\label{deriv-v}
v^{(i)}(t)=\int_0^t \frac{\partial^i \!c_{_X}}{\partial t^i}(s,t)\,u(s)\dd s.
\end{equation}
Finally, due to~(\ref{condition-deriv}),
\begin{align*}
\mathcal{D} v(t)
&
=\sum_{k=0}^{n-1}\left[\mathcal{D}\psi_k(t) \int_0^t \varphi_k(s)u(s)\,\dd s
+\mathcal{D}\varphi_k(t) \int_t^1 \psi_k(s)u(s)\,\dd s\right]+\kappa\alpha_p u(t)
=\kappa\alpha_p u(t).
\end{align*}

Concerning the boundary value conditions, referring to~(\ref{condition-deriv}),
we similarly have
$$
(\mathcal{D}_0 v)(0)=\sum_{k=0}^{n-1}(\mathcal{D}_0 \varphi_k)(0)
\int_0^1 \psi_k(s)u(s)\,\dd s=0,\quad
(\mathcal{D}_1 v)(1)=\sum_{k=0}^{n-1}(\mathcal{D}_1 \psi_k)(1)
\int_0^1 \varphi_k(s)u(s)\,\dd s=0.
$$
The proof of Theorem~\ref{th-gene} is finished.
\qed
\end{proof}

In the two next sections, we construct processes connected to
the equation $\mathcal{D} v=u$ subject to the boundary value conditions at $0$:
$(\mathcal{D}^i_0 v)(0)=0$ for $i\in\{0,1,\dots,n-1\}$ and others at $1$ that
will be discussed subsequently, where $\mathcal{D}$ and $\mathcal{D}^i_0$ are
the differential operators ($\mathcal{D}$ being of order $p=2n$) defined by
$$
\mathcal{D}=(-1)^n\frac{\dd^{2n}}{\dd t^{2n}},\quad
\mathcal{D}^i_0=\frac{\dd^{i}}{\dd t^{i}}.
$$
\section{The $(n-1)$-fold integral of Brownian motion}\label{sect-iteratedIBM}

Let $(B(t))_{t\in [0,1]}$ be the linear Brownian motion limited to the time
interval $[0,1]$ and started at $0$. We introduce the $(n-1)$-fold integral
of Brownian motion: for any $t\in[0,1]$,
$$
X_n(t)=\int_0^t \frac{(t-s)^{n-1}}{(n-1)!} \,\dd B(s).
$$
In particular, $X_1=B$.
The trajectories of $(X_n(t))_{t\in [0,1]}$ are $n$ times differentiable
and we have $X_n^{(i)}=X_{n-i}$ for $0\le i\le n-1$.
Moreover, we have at time~$0$ the equalities $X_n(0)=X_{n-1}(0)=\dots=X_2(0)=X_1(0)=0$.
The process $(X_n(t))_{t\in [0,1]}$ is a $n$-Markov Gaussian process
since the $n$-dimensional process $(X_n(t),X_{n-1}(t),\dots,X_1(t))_{t\in [0,1]}$
is Markovian. The covariance function of the Gaussian process
$(X_n(t))_{t\in [0,1]}$ is given by
$$
c_{_{X_n}}(s,t)=\int_0^{s\wedge t} \frac{(s-u)^{n-1}}{(n-1)!}
\,\frac{(t-u)^{n-1}}{(n-1)!} \,\dd u.
$$
In order to apply Theorem~\ref{th-gene}, we decompose $c_{_{X_n}}$ into
the form~(\ref{cov-gene}). We have for, e.g., $s\le t$,
\begin{align*}
c_{_{X_n}}(s,t)
&
=\int_0^s \frac{(s-u)^{n-1}}{(n-1)!^2}\left[\,\sum_{k=0}^{n-1}\binom{n-1}{k}
(-u)^{n-1-k}t^k\right] \dd u
\\
&
=\sum_{k=0}^{n-1} (-1)^{n-1-k}\,\frac{t^k}{k!} \int_0^s \frac{(s-u)^{n-1}}{(n-1)!}\,
\frac{u^{n-1-k}}{(n-1-k)!}\, \dd u
\\
&
=\sum_{k=0}^{n-1} (-1)^{n-1-k}\,\frac{s^{2n-1-k}}{(2n-1-k)!}\,\frac{t^k}{k!}.
\end{align*}
We then obtain the following representation:
$$
c_{_{X_n}}(s,t)=\sum_{k=0}^{n-1} \varphi_k(s) \psi_k(t)
$$
with, for any $k\in\{0,1,\dots,n-1\}$,
$$
\varphi_k(s)=(-1)^{n-1-k}\,\frac{s^{2n-1-k}}{(2n-1-k)!},\quad
\psi_k(t)=\frac{t^k}{k!}.
$$
We state below a result of~\cite{bridge} that we revisit here by using
Theorem~\ref{th-gene}.
%
\begin{theorem}\label{th-BVP-IBM}
Let $u$ be a fixed continuous function on $[0,1]$.
The function $v$ defined on $[0,1]$ by
$$
v(t)=\int_0^1 c_{_{X_n}}(s,t)u(s)\,\dd s
$$
is the solution of the boundary value problem
\begin{equation}\label{BVP-IBM}
\begin{cases}
v^{(2n)}=(-1)^n u &\mbox{on }[0,1],
\\
v^{(i)}(0)=0 &\mbox{for } i\in\{0,1,\dots,n-1\},
\\
v^{(i)}(1)=0 &\mbox{for } i\in\{n,n+1,\dots,2n-1\}.
\end{cases}
\end{equation}
\end{theorem}
%
\begin{proof}
Let us check that the conditions~(\ref{condition-sum}) and~(\ref{condition-deriv})
of Theorem~\ref{th-gene} are fulfilled. First, we have
\begin{align*}
\lqn{\sum_{k=0}^{n-1}\left[\varphi_k(t)\psi_k^{(i)}(t)
-\varphi_k^{(i)}(t)\psi_k(t)\right]}
&
=\sum_{k=0}^{n-1} (-1)^{n-1-k}\left[\ind_{\{k\ge i\}}
\,\frac{t^{2n-1-k}}{(2n-1-k)!}\,\frac{t^{k-i}}{(k-i)!}
-\ind_{\{k\le 2n-1-i\}}\,\frac{t^{2n-1-i-k}}{(2n-1-i-k)!}\,\frac{t^k}{k!}\right]
\\
&
=(-1)^{n-1}\,\frac{t^{2n-1-i}}{(2n-1-i)!}
\left[\vphantom{\sum_n^n}\right.\!\ind_{\{i\le n-1\}}\sum_{k=i}^{n-1}
(-1)^{k}\binom{2n-1-i}{k-i}
-\sum_{k=0}^{(2n-1-i)\wedge(n-1)} (-1)^{k}\binom{2n-1-i}{k}\!\!\left.
\vphantom{\sum_n^n}\right]
\\
&
=(-1)^{n-1}\,\frac{t^{2n-1-i}}{(2n-1-i)!}
\left[\ind_{\{i\le n-1\}}\left(\,\sum_{k=0}^{n-1-i} (-1)^{i+k}\binom{2n-1-i}{k}
-\sum_{k=0}^{n-1} (-1)^{k}\binom{2n-1-i}{k}\right)\!\!\right.
\\
&
\hphantom{=\,}\left.-\ind_{\{i\ge n\}}\sum_{k=0}^{2n-1-i}
(-1)^{k}\binom{2n-1-i}{k}\!\right]\!.
\end{align*}
Performing the transformation $k\mapsto 2n-1-i-k$ in the first sum lying
within the last equality, we get
\begin{align*}
\lqn{\sum_{k=0}^{n-1-i} (-1)^{i+k}\binom{2n-1-i}{k}
-\sum_{k=0}^{n-1} (-1)^{k}\binom{2n-1-i}{k}}
&
=\sum_{k=n}^{2n-1-i} (-1)^{k}\binom{2n-1-i}{2n-1-i-k}
+\sum_{k=0}^{n-1} (-1)^{k}\binom{2n-1-i}{k}
=\sum_{k=0}^{2n-1-i} (-1)^k\binom{2n-1-i}{k}=\delta_{i,2n-1}
\end{align*}
and then
$$
\sum_{k=0}^{n-1}\left[\varphi_k(t)\psi_k^{(i)}(t)-\varphi_k^{(i)}(t)\psi_k(t)\right]
=(-1)^n\delta_{i,2n-1}.
$$
On the other hand, setting
$$
\mathcal{D}=(-1)^n\frac{\dd^{2n}}{\dd t^{2n}},\quad\mathcal{D}_0^{i}
=\frac{\dd^i}{\dd t^i},\quad
\mathcal{D}_1^{i}=\frac{\dd^{i+n}}{\dd t^{i+n}},
$$
we clearly see that, for any $k\in\{0,1,\dots,n-1\}$,
$$
\mathcal{D} \varphi_k=\mathcal{D} \psi_k=0\quad\mbox{and}\quad\mathcal{D}_0^i
\varphi_k(0)=\mathcal{D}_1^i \psi_k(1)=0\quad\mbox{for } i\in\{0,1,\dots,n-1\}.
$$
Consequently, by Theorem~\ref{th-gene}, we see that the function $v$ solves the
boundary value problem~(\ref{BVP-IBM}). The uniqueness part
will follow from a more general argument stated in the proof of
Theorem~\ref{th-BVP-gene}.
\qed
\end{proof}

\section{Various bridges of the $(n-1)$-fold integral of Brownian motion}\label{sect-bridges}

In this section, we construct various bridges related to $(X_n(t))_{t\in [0,1]}$.
More precisely, we take $(X_n(t))_{t\in [0,1]}$ conditioned on the event
that certain derivatives vanish at time~$1$.
Let us recall that all the derivatives at time~$0$ naturally vanish:
$X_n(0)=X_{n-1}(0)=\dots=X_2(0)=X_1(0)=0$.

For any $m\in\{0,1,\dots,n\}$, let $J=\{j_1,j_2,\dots,j_m\}$ be a subset of
$\{1,2,\dots,n\}$ with $1\le j_1<j_2<\dots<j_m\le n$ and the convention that
for $m=0$, $J=\varnothing$. We see that for each $m\in\{0,1,\dots,n\}$, we can
define $\binom{n}{m}$ subsets of indices $J$, and the total number of sets $J$
is then $\sum_{m=0}^n \binom{n}{m}=2^n$. Set for any $t\in[0,1]$
$$
Y(t)=(X_n(t) | X_j(1)=0,j\in J)=
(X_n(t)|X_{j_1}(1)=X_{j_2}(1)=\dots=X_{j_m}(1)=0).
$$
In this way, we define $2^n$ processes $(Y(t))_{t\in[0,1]}$ that we shall
call ``bridges'' of $(X_n(t))_{t\in [0,1]}$. In particular,
\begin{itemize}
\item
for $J=\varnothing$, we simply have $(Y(t))_{t\in[0,1]}=(X_n(t))_{t\in[0,1]}$;
\item
for $J=\{1\}$, the corresponding process is the $(n-1)$-fold integral of Brownian bridge
$$
(X_n(t) | X_1(1)=0)_{t\in[0,1]}=\left(\int_0^t \frac{(t-s)^{n-1}}{(n-1)!}\,
\dd \beta(s)\right)_{t\in[0,1]};
$$
\item
for $J=\{n\}$, the corresponding process is the ``single'' bridge of
$(n-1)$-fold integral of Brownian Brownian:
$$
(X_n(t) | X_n(1)=0)_{t\in[0,1]}=\left(\int_0^t \frac{(t-s)^{n-1}}{(n-1)!}\,\dd B(s)
\,\bigg| \int_0^1 \frac{(1-s)^{n-1}}{(n-1)!}\,\dd B(s)=0\right)_{t\in[0,1]};
$$
\item
for $J=\{1,2,\dots,n\}$, the corresponding process is
$$(X_n(t) | X_n(1)=X_{n-1}(1)=\dots=X_1(1)=0)_{t\in[0,1]}.$$
This is the natural bridge related to the $n$-dimensional Markov process
$(X_n(t),X_{n-1}(t),\dots,\linebreak X_1(t))_{t\in[0,1]}.$
\end{itemize}

In this section, we exhibit several interesting properties of the various processes
$(Y(t))_{t\in [0,1]}$.
One of the main goals is to relate these bridges to additional
boundary value conditions at $1$. For this, we introduce the following subset $I$
of $\{0,1,\dots,2n-1\}$:
$$
I=(n-J)\cup\left[\{n,n+1,\dots,2n-1\}\backslash(J+n-1)\right]
$$
with
\begin{align*}
n-J
&
=\{n-j,j\in J\}=\{n-j_1,\dots,n-j_m\},
\\
J+n-1
&
=\{j+n-1,j\in J\}=\{j_1+n-1,\dots,j_m+n-1\}.
\end{align*}
The cardinality of $I$ is $n$. Actually, the set $I$ will be used further for enumerating
the boundary value problems which can be related to the bridges labeled by $J$.
Conversely, $I$ yields $J$ through
$$
J=(n-I)\cap\{1,2,\dots,n\}.
$$
In the table below, we give some examples of sets $I$ and $J$.
$$
\begin{array}{|@{\hspace{.3em}}c@{\hspace{.3em}}|@{\hspace{.3em}}c@{\hspace{.3em}}|
@{\hspace{.3em}}c@{\hspace{.3em}}|@{\hspace{.3em}}c@{\hspace{.3em}}|
@{\hspace{.3em}}c@{\hspace{.3em}}|}
\hline\vspace{-2ex}&&&&
\\
I & \{0,1,\dots,n-1\} & \{n,n+1,\dots,2n-1\} & \{n-1,n+1,\dots,2n-1\} & \{0,n+1,\dots,2n-2\}
\\
\vspace{-2ex}&&&&
\\
\hline\vspace{-2ex}&&&&
\\
J & \{1,2,\dots,n\} & \varnothing & \{1\} & \{n\}
\\[-2ex]
&&&&
\\
\hline
\end{array}
$$

\subsection{Polynomial drift description}

Below, we provide a representation of $(Y(t))_{t\in [0,1]}$ by means of
$(X_n(t))_{t\in [0,1]}$ subject to a random polynomial drift.
%
\begin{theorem}\label{th-drift}
We have the distributional identity
$$
(Y(t))_{t\in[0,1]}\stackrel{d}{=}\bigg(X_n(t)-\sum_{j\in J} P_j(t)X_j(1)\bigg)_{t\in[0,1]}
$$
where the functions $P_j$, $j\in J$, are Hermite interpolation polynomials on $[0,1]$
characterized by
\begin{equation}\label{condition-hermite}
\begin{cases}
P_j^{(2n)}=0,
\\[1ex]
P_j^{(i)}(0)=0 &\mbox{for } i\in\{0,1,\dots,n-1\},
\\[1ex]
P_j^{(i)}(1)=\delta_{j,n-i} &\mbox{for } i\in I.
\end{cases}
\end{equation}
\end{theorem}
%
\begin{remark}
In the case where $n=2$, we retrieve a result of~\cite{drift}.
Moreover, the conditions~(\ref{condition-hermite}) characterize the
polynomials $P_j$, $j\in J$. We prove this fact in Lemma~\ref{hermite} in the appendix.
\end{remark}
%
\begin{proof}
By invoking classical arguments of Gaussian processes theory, we have the
distributional identity
$$
(Y(t))_{t\in[0,1]}\stackrel{d}{=}\bigg(X_n(t)-\sum_{j\in J} P_j(t)X_j(1)\bigg)_{t\in[0,1]}
$$
where the functions $P_j$, $j\in J$, are such that $\mathbb{E}[Y(t)X_k(1)]=0$
for all $k\in J$. We get the linear system
\begin{equation}\label{system}
\sum_{j\in J} \mathbb{E}[X_j(1)X_k(1)]\,P_j(t)=\mathbb{E}[X_n(t)X_k(1)],\quad k\in J.
\end{equation}
We plainly have
$$
\mathbb{E}[X_j(s)X_k(t)]=\int_0^{s\wedge t} \frac{(s-u)^{j-1}}{(j-1)!}\,
\frac{(t-u)^{k-1}}{(k-1)!}\,\dd u.
$$
Then, the system~(\ref{system}) writes
$$
\sum_{j\in J} \frac{1}{j+k-1}\,\frac{P_j(t)}{(j-1)!}
=\int_0^t \frac{(t-u)^{n-1}}{(n-1)!}\,(1-u)^{k-1}\,\dd u.
$$
The matrix of this system $\left(1/(j+k-1)\right)_{j,k\in J}$ is regular as
it can be seen by introducing the related quadratic form
which is definite positive. Indeed, for any real numbers $x_j$, $j\in J$, we have
$$
\sum_{j,k\in J} \frac{x_jx_k}{j+k-1}
=\int_0^1 \left(\vphantom{\sum_n^n}\right.\!\sum_{j,k\in J}
x_jx_k u^{j+k-2}\!\!\left.\vphantom{\sum_n^n}\right)\dd u
=\int_0^1 \left(\vphantom{\sum_n^n}\right.\!\sum_{j\in J}
x_j u^{j-1}\!\!\left.\vphantom{\sum_n^n}\right)^{\!2}\dd u\ge 0
$$
and
$$
\sum_{j,k\in J} \frac{x_jx_k}{j+k-1}=0\quad\mbox{if and only if}
\quad \forall j\in J,\, x_j=0.
$$
Thus, the system~(\ref{system}) has a unique solution. As a result, the $P_j$
are linear combinations of the functions $t\mapsto\int_0^t (t-u)^{n-1}(1-u)^{k-1}\,\dd u$
which are polynomials of degree less than $n+k$.
Hence, $P_j$ is a polynomial of degree at most $2n-1$.

We now compute the derivatives of $P_j$ at $0$ en $1$.
We have $P_j^{(i)}(0)=0$ for $i\in\{0,1,\dots,n-1\}$ since the functions
$t\mapsto\int_0^t (t-u)^{n-1}(1-u)^{k-1}\,\dd u$ plainly enjoy this property.
For checking that $P_j^{(i)}(1)=\delta_{j,n-i}$ for $i\in I$,
we successively compute
\begin{itemize}
\item
for $i\in\{0,1,\dots,n-1\}$,
$$
\frac{\dd^i}{\dd t^i} \,\mathbb{E}[X_n(t)X_k(1)]=\mathbb{E}[X_{n-i}(t)X_k(1)];
$$
\item
for $i=n-1$,
$$
\frac{\dd^i}{\dd t^i}\,\mathbb{E}[X_n(t)X_k(1)]=\mathbb{E}[B(t)X_k(1)]
=\int_0^t \frac{(1-u)^{k-1}}{(k-1)!}\,\dd u;
$$
\item
for $i=n$,
$$
\frac{\dd^i}{\dd t^i}\,\mathbb{E}[X_n(t)X_k(1)]=\frac{(1-t)^{k-1}}{(k-1)!};
$$
\item
for $i\in\{n,n+1,\dots,n+k-1\}$,
$$
\frac{\dd^i}{\dd t^i}\,\mathbb{E}[X_n(t)X_k(1)]
=(-1)^{i+n}\,\frac{(1-t)^{k+n-i-1}}{(k+n-i-1)!};
$$
\item
for $i=k+n-1$,
$$
\frac{\dd^i}{\dd t^i}\,\mathbb{E}[X_n(t)X_k(1)]=(-1)^{k-1};
$$
\item
for $i\ge k+n$,
$$
\frac{\dd^i}{\dd t^i}\,\mathbb{E}[X_n(t)X_k(1)]=0.
$$
\end{itemize}
Consequently, at time~$t=1$, for $i\ge n$,
\begin{equation}\label{deriv-cov}
\frac{\dd^i}{\dd t^i}\,\mathbb{E}[X_n(t)X_k(1)]\Big|_{t=1}=(-1)^{k-1}\delta_{i,k+n-1}.
\end{equation}
Now, by differentiating~(\ref{system}), we get for $i\in\{0,1,\dots,n-1\}$,
$$
\sum_{j\in J} \mathbb{E}[X_j(1)X_k(1)] \, P_j^{(i)}(1)=\mathbb{E}[X_{n-i}(1)X_k(1)],\quad k\in J.
$$
In particular, if $i\in (n-J)$, this can be rewritten as
$$
\sum_{j\in J} \mathbb{E}[X_j(1)X_k(1)]\,P_j^{(i)}(1)
=\sum_{j\in J} \delta_{j,n-i}\,\mathbb{E}[X_j(1)X_k(1)],\quad k\in J,
$$
which by identification yields $P_j^{(i)}(1)=\delta_{j,n-i}$.
Similarly, for $i\in\{n,n+1,\dots,2n-1\}$, in view of~(\ref{deriv-cov}), we have
$$
\sum_{j\in J} \mathbb{E}[X_j(1)X_k(1)]\,P_j^{(i)}(1)
=(-1)^{k-1}\delta_{i,k+n-1},\quad k\in J.
$$
In particular, if $i\in\{n,n+1,\dots,2n-1\}\backslash(J+n-1)$, we have
$\delta_{i,k+n-1}=0$ for $k\in J$, and then
$$
\sum_{j\in J} \mathbb{E}[X_j(1)X_k(1)]\,P_j^{(i)}(1)=0,\quad k\in J,
$$
which by identification yields $P_j^{(i)}(1)=0=\delta_{j,n-i}$.
The proof of Theorem~\ref{th-drift} is finished.
\qed
\end{proof}

\subsection{Covariance function}

Let $c_{_Y}$ be the covariance function of $(Y(t))_{t\in[0,1]}$:
$c_{_Y}(s,t)=\mathbb{E}[Y(s)Y(t)]$.
In the next theorem, we supply a representation of $c_{_Y}$ of the form~(\ref{cov-gene}).
%
\begin{theorem}\label{th-decompo-cov}
The covariance function of $(Y(t))_{t\in[0,1]}$ admits the following representation:
for $s,t\in[0,1]$,
$$
c_{_Y}(s,t)=\sum_{k=0}^{n-1} \varphi_k(s\wedge t) \tilde{\psi}_k(s\vee t)
$$
with, for any $k\in\{0,1,\dots,n-1\}$,
$$
\tilde{\psi}_k(t)=\psi_k(t)-\sum_{j\in J}\psi_k^{(n-j)}(1) P_j(t).
$$
Moreover, the functions $\tilde{\psi}_k$, $k\in\{0,1,\dots,n-1\}$, are Hermite
interpolation polynomials such that
$$
\begin{cases}
\tilde{\psi}_k^{(2n)}=0,
\\[1ex]
\tilde{\psi}_k^{(i)}(0)=\delta_{i,k} &\mbox{for }i\in\{0,1,\dots,n-1\},
\\[1ex]
\tilde{\psi}_k^{(i)}(1)=0 &\mbox{for }i\in I.
\end{cases}
$$
\end{theorem}
%
\begin{proof}
We decompose $Y(t)$ into the difference
$$
Y(t)=X_n(t)-Z(t) \quad\mbox{ with }\quad Z(t)=\sum_{j\in J} P_j(t)X_j(1).
$$
We have
\begin{align*}
c_{_Y}(s,t)
&
=\mathbb{E}[X_n(s)X_n(t)]+\mathbb{E}[Z(s)Z(t)]-\mathbb{E}[X_n(s)Z(t)]-\mathbb{E}[Z(s)X_n(t)]
\\
&
=c_{_{X_n}}(s,t)+\sum_{j,k\in J} \mathbb{E}[X_j(1)X_k(1)]\,P_j(s)P_k(t)
\\
&
\hphantom{=\,}-\sum_{j\in J} \mathbb{E}[X_n(t)X_j(1)]\,P_j(s)
-\sum_{j\in J} \mathbb{E}[X_n(s)X_j(1)]\,P_j(t).
\end{align*}
By definition~(\ref{system}) of the $P_j$'s, we observe that
\begin{align*}
\lqn{\sum_{j,k\in J} \mathbb{E}[X_j(1)X_k(1)]\,P_j(s)P_k(t)
-\sum_{k\in J} \mathbb{E}[X_n(s)X_k(1)]\,P_k(t)}
&
=\sum_{k\in J} \left(\vphantom{\sum_n^n}\right.\!\sum_{j\in J}
\mathbb{E}[X_j(1)X_k(1)]\,P_j(s)-\mathbb{E}[X_n(s)X_k(1)]
\!\!\left.\vphantom{\sum_n^n}\right)\!P_k(t)
=0.
\end{align*}
Then, we can simplify $c_{_Y}(s,t)$ into
$$
c_{_Y}(s,t)=c_{_{X_n}}(s,t)-\sum_{j\in J} \mathbb{E}[X_n(t)X_j(1)]\,P_j(s).
$$
Since the covariance functions $c_{_Y}$ and $c_{_{X_n}}$ are symmetric,
we also have
$$
c_{_Y}(s,t)=c_{_{X_n}}(s,t)-\sum_{j\in J} \mathbb{E}[X_n(s)X_j(1)]\,P_j(t).
$$
Let us introduce the symmetric polynomial
$$
Q(s,t)=\sum_{j\in J} \mathbb{E}[X_n(s)X_j(1)]\,P_j(t)
=\sum_{j\in J} \mathbb{E}[X_n(t)X_j(1)]\,P_j(s).
$$
It can be expressed by means of the functions $\varphi_k,\psi_k$'s as follows:
$$
Q(s,t)=\sum_{j\in J}\sum_{k=0}^{n-1} \varphi_k(s\wedge t) \psi_k^{(n-j)}(1) P_j(s\vee t).
$$
We can rewrite $c_{_Y}(s,t)$ as
$$
c_{_Y}(s,t)=c_{_{X_n}}(s,t)-Q(s,t)
=\sum_{k=0}^{n-1} \varphi_k(s\wedge t) \!\left[\vphantom{\sum_n^n}\right.\!
\psi_k(s\vee t)-\sum_{j\in J}\psi_k^{(n-j)}(1) P_j(s\vee t)\!\!\left.
\vphantom{\sum_n^n}\right]\!.
$$
and then, for $s,t\in[0,1]$,
$$
c_{_Y}(s,t)=\sum_{k=0}^{n-1} \varphi_k(s\wedge t) \tilde{\psi}_k(s\vee t)
\quad\mbox{with}\quad
\tilde{\psi}_k(t)=\psi_k(t)-\sum_{j\in J}\psi_k^{(n-j)}(1) P_j(t).
$$
We immediately see that $\tilde{\psi}_k$ is a polynomial of degree less than $2n$ such that
$\tilde{\psi}_k^{(i)}(0)=\delta_{i,k}$ for $i\in\{0,1,\dots,n-1\}$
and, since $P_j^{(i)}(1)=\delta_{j,n-i}$,
$$
\tilde{\psi}_k^{(i)}(1)=\psi_k^{(i)}(1)-\sum_{j\in J}\psi_k^{(n-j)}(1) P_j^{(i)}(1)
=\left(1-\ind_{\{i\in n-J\}}\right)\psi_k^{(i)}(1).
$$
We deduce that
$$
\begin{cases}
\tilde{\psi}_k^{(i)}(1)=0 &\mbox{if }i\in (n-J),
\\[.5ex]
\tilde{\psi}_k^{(i)}(1)=\psi_k^{(i)}(1)=0 &\mbox{if } i\in\{n,n+1,\dots,2n-1\}\backslash (J+n-1).
\end{cases}
$$
Then $\tilde{\psi}_k^{(i)}(1)=0$ for any $i\in I$. This ends up the proof of
Theorem~\ref{th-decompo-cov}.
\qed
\end{proof}

\subsection{Boundary value problem}

In this section, we write out the natural boundary value problem which is
associated with the process $(Y(t))_{t\in[0,1]}$.
The following statement is the main connection between the different
boundary value conditions associated with the operator $d^{2n}/dt^{2n}$
and the different bridges of the process $(X_n)_{t\in[0,1]}$ introduced
in this work.
%
\begin{theorem}\label{th-BVP-gene}
Let $u$ be a fixed continuous function on $[0,1]$. The function $v$ defined on $[0,1]$ by
$$
v(t)=\int_0^1 c_{_Y}(s,t)u(s)\,\dd s
$$
is the solution of the boundary value problem
\begin{equation}\label{BVP-bridge}
\begin{cases}
v^{(2n)}=(-1)^n u &\mbox{on }[0,1],
\\
v^{(i)}(0)=0 &\mbox{for } i\in\{0,1,\dots,n-1\},
\\
v^{(i)}(1)=0 &\mbox{for } i\in I.
\end{cases}
\end{equation}
\end{theorem}
%
\begin{proof}
\noindent$\bullet$ \textsl{First step.}
Recall that $c_{_Y}(s,t)=c_{_{X_n}}(s,t)-Q(s,t)$.
We decompose the function $v$ into the difference $w-z$ where, for $t\in[0,1]$,
$$
w(t)=\int_0^1 c_{_{X_n}}(s,t)u(s)\,\dd s,\quad z(t)=\int_0^1 Q(s,t)u(s)\,\dd s.
$$
We know from Theorem~\ref{th-BVP-IBM} that $w^{(2n)}=(-1)^nu$,
$w^{(i)}(0)=\frac{\partial^i\!Q}{\partial t^i}(s,0)=0$
for $i\in\{0,1,\dots,n-1\}$ and $w^{(i)}(1)=0$ for $i\in\{n,n+1,\dots,2n-1\}$.
Moreover, the function $t\mapsto Q(s,t)$ being a polynomial of
degree less than $2n$, the function $z$ is also a polynomial of degree less than $2n$.
Then $z^{(2n)}=0$, $z^{(i)}(0)=0$ for $i\in\{0,1,\dots,n-1\}$ and
$$
v^{(2n)}=w^{(2n)}-z^{(2n)}=(-1)^nu,\quad
v^{(i)}(0)=w^{(i)}(0)-z^{(i)}(0)=0 \quad\mbox{for }i\in\{0,1,\dots,n-1\}.
$$
On the other hand, we learn from~(\ref{deriv-v}) that, for $i\in\{0,1,\dots,2n-1\}$,
$$
w^{(i)}(1)=\int_0^1 \frac{\partial^i \!c_{_{X_n}}}{\partial t^i}(s,1)\,u(s)\dd s
$$
and then, for $i\in\{0,1,\dots,n-1\}$,
$$
w^{(i)}(1)=\int_0^1 \mathbb{E}[X_n(s)X_{n-i}(1)]\,u(s)\,\dd s.
$$
We also have, for any $i\in\{0,1,\dots,2n-1\}$,
$$
\frac{\partial^i\!Q}{\partial t^i}(s,1)=\sum_{j\in J} \mathbb{E}[X_n(s)X_j(1)]\,P^{(i)}_j(1)
=\ind_{\{i\in (n-J)\}}\mathbb{E}[X_n(s)X_{n-i}(1)].
$$
As a result, we see that
$$
z^{(i)}(1)=\int_0^1 \frac{\partial^i\!Q}{\partial t^i}(s,1)u(s)\,\dd s
=\ind_{\{i\in (n-J)\}}w^{(i)}(1).
$$
This implies that for $i\in (n-J)$, $z^{(i)}(1)=w^{(i)}(1)$ and
for $i\in\{n,n+1,\dots,2n-1\}\backslash(J+n-1)$, $z^{(i)}(1)=0=w^{(i)}(1)$.
Then $z^{(i)}(1)=w^{(i)}(1)$ for any $i\in I$, that is $v^{(i)}(1)=0$.
The function $v$ is a solution of~(\ref{BVP-bridge}).

\noindent$\bullet$ \textsl{Second step.}
We now check the uniqueness of the solution of~(\ref{BVP-bridge}).
Let $v_1$ and $v_2$ be two solutions of $\mathcal{D} v=u$. Then the function $w=v_1-v_2$
satisfies $\mathcal{D} w=0$, $w^{(i)}(0)=0$ for $i\in\{0,1,\dots,n-1\}$
and $w^{(i)}(1)=0$ for $i\in I$. We compute the following ``energy'' integral:
\begin{align*}
\int_0^1 w^{(n)}(t)^2\,\dd t
&
=(-1)^{n+1}\left[\,\sum_{i=0}^{n-1} (-1)^{i} w^{(i)}(t)w^{(2n-1-i)}(t)
\right]_0^1+(-1)^n\int_0^1 w(t)w^{(2n)}(t)\,\dd t
\\
&
=(-1)^{n+1}\sum_{i=0}^{n-1} (-1)^{i} w^{(i)}(1)w^{(2n-1-i)}(1).
\end{align*}
We have constructed the set $I$ in order to have $w^{(i)}(1)w^{(2n-1-i)}(1)=0$
for any $i\in\{0,1,\dots,n-1\}$: when we pick an index $i\in\{0,1,\dots,n-1\}$,
either $i\in I$ or $2n-1-i\in I$. Indeed,
\begin{itemize}
\item[--]
if $i\in  I\cap\{0,1,\dots,n-1\}$, $w^{(i)}=0$;

\item[--]
if $i\in \{0,1,\dots,n-1\}\backslash I$, by observing that
$\{0,1,\dots,n-1\}\backslash I=\{0,1,\dots,n-1\}\backslash (n-J)$,
we have $2n-1-i\in \{n,n+1,\dots,2n-1\}\backslash (J+n-1)$.
Since $\{n,n+1,\dots,2n-1\}\backslash (J+n-1)\subset I$, we see that
$2n-1-i\in I$ and then $w^{(2n-1-i)}=0$.
\end{itemize}
Next, $\int_0^1 w^{(n)}(t)^2\,\dd t=0$ which entails $w^{(n)}=0$, that is,
$w$ is a polynomial of degree less than $n$. Moreover, with the
boundary value conditions at $0$, we obtain $w=0$ or $v_1=v_2$.

The proof of Theorem~\ref{th-BVP-gene} is finished.
\qed
\end{proof}
%
\begin{remark}\label{remark-unique}
We have seen in the above proof that uniqueness is assured as soon as
the boundary value conditions at $1$ satisfy $w^{(i)}(1)w^{(2n-1-i)}(1)=0$
for any $i\in\{0,1,\dots,n-1\}$. These conditions are fulfilled when
the set $I=\{i_1,i_2,\dots,i_n\}\subset\{0,1,\dots,2n-1\}$ is such that
$i_1\in\{0,2n-1\}$, $i_2\in\{1,2n-2\}$, $i_3\in\{2,2n-3\}$, $\dots$,
$i_n\in\{n-1,n\}$. This is equivalent to say that $I$ and $(2n-1-I)$
make up a partition of $\{0,1,\dots,2n-1\}$, or
$$
2n-1-I=\{0,1,\dots,2n-1\}\backslash I.
$$
In this manner, we get $2^n$ different boundary value problems
which correspond to the $2^n$ different bridges we have constructed.
We shall see in Section~\ref{sect-general} that the above identity
concerning the differentiating set $I$
characterizes the possibility for the Green function of the boundary
value problem~(\ref{BVP-bridge}) to be the covariance of a Gaussian process.
\end{remark}

\subsection{Prediction}

Now, we tackle the problem of the prediction for the process $(Y(t))_{t\in[0,1]}$.
%
\begin{theorem}\label{th-prediction}
Fix $t_0\in[0,1]$. The shifted process $(Y(t+t_0))_{t\in[0,1-t_0]}$ admits the
following representation:
$$
(Y(t+t_0))_{t\in[0,1-t_0]}=\left(\tilde{Y}_{t_0}(t)+\sum_{i=0}^{n-1} Q_{i,t_0}(t) Y^{(i)}(t_0)
\right)_{t\in[0,1-t_0]}
$$
where
$$
\tilde{Y}_{t_0}(t)=\tilde{X}_n(t)-\sum_{j\in J}\tilde{P}_{j,t_0}(t)\tilde{X}_j(1-t_0).
$$
The process $\big(\tilde{X}_n(t)\big)_{t\in[0,1-t_0]}$ is a copy of $(X_n(t))_{t\in[0,1-t_0]}$
which is independent of $(X_n(t))_{t\in[0,t_0]}$.

The functions $\tilde{P}_{j,t_0}$, $j\in J$, and $Q_{i,t_0}$, $i\in\{0,1,\dots,n-1\}$,
are Hermite interpolation polynomials on $[0,1-t_0]$ characterized by
$$
\begin{cases}
\tilde{P}^{(2n)}_{j,t_0}=0,
\\[1ex]
\tilde{P}^{(\iota)}_{j,t_0}(0)=0
&\mbox{for } \iota\in\{0,1,\dots,n-1\},
\\[1ex]
\tilde{P}^{(\iota)}_{j,t_0}(1-t_0)=\delta_{\iota,n-j}&\mbox{for }\iota\in I,
\end{cases}
$$
and
$$
\begin{cases}
Q^{(2n)}_{i,t_0}=0,
\\[1ex]
Q^{(\iota)}_{i,t_0}(0)=\delta_{\iota, i}
&\mbox{for } \iota\in\{0,1,\dots,n-1\},
\\[1ex]
Q^{(\iota)}_{i,t_0}(1-t_0)=0 &\mbox{for }\iota\in I.
\end{cases}
$$
Actually, these functions can be expressed by means of
the functions $P_j$, $j\in J$, and $\tilde{\psi}_i$, $i\in\{0,1,\dots,n-1\}$, as follows:
$$
\tilde{P}_{j,t_0}(t)=(1-t_0)^{n-j}P_j\!\left(\frac{t}{1-t_0}\right)\!,
\quad Q_{i,t_0}=(1-t_0)^{i}\tilde{\psi}_i\!\left(\frac{t}{1-t_0}\right)\!.
$$
In other words, the process $\big(\tilde{Y}_{t_0}(t)\big)_{t\in[0,1-t_0]}$ is
a bridge of length $(1-t_0)$ which is independent of $(Y(t))_{t\in[0,t_0]}$, that is
$$
\big(\tilde{Y}_{t_0}(t)\big)_{t\in[0,1-t_0]}\stackrel{d}{=}
\big(\tilde{X}_n(t) \big| \tilde{X}_j(1-t_0)=0,j\in J\big)_{t\in[0,1-t_0]}.
$$
\end{theorem}
%
\begin{proof}
\noindent$\bullet$ \textsl{First step.}
Fix $t_0\in[0,1]$. We have the well-known decomposition, based on the classical prediction
property of Brownian motion stipulating that
$(B(t+t_0))_{t\in[0,1-t_0]}=\big(\tilde{B}_{t_0}(t)+B(t_0)\big)_{t\in[0,1-t_0]}$
where $\big(\tilde{B}_{t_0}(t)\big)_{t\in[0,1-t_0]}$ is a Brownian motion
independent of $(B(t))_{t\in[0,t_0]}$,
$$
(X_n(t+t_0))_{t\in[0,1-t_0]}=\left(\tilde{X}_n(t)+\sum_{k=0}^{n-1}
\frac{t^k}{k!} \, X_{n-k}(t_0)\right)_{t\in[0,1-t_0]}
$$
with $\tilde{X}_n(t)=\int_0^1 \frac{(t-s)^{n-1}}{(n-1)!}\,\dd \tilde{B}(s)$.
Differentiating this equality $(n-j)$ times, $j\in\{1,2,\dots,n\}$, we obtain
$$
X_j(t+t_0)=\tilde{X}_j(t)+\sum_{k=n-j}^{n-1} \frac{t^{j+k-n}}{(j+k-n)!} \, X_{n-k}(t_0).
$$
Therefore,
\begin{align}
\lqn{Y(t+t_0)}
&
=X_n(t+t_0)-\sum_{j\in J} P_j(t+t_0)X_j(1)
\nonumber\\
&
=\tilde{X}_n(t)+\sum_{k=0}^{n-1} \frac{t^k}{k!} \,X_{n-k}(t_0)
-\sum_{j\in J} P_j(t+t_0) \!\left[\vphantom{\sum_n^n}\right.\!\tilde{X}_j(1-t_0)
+\sum_{k=n-j}^{n-1} \frac{(1-t_0)^{j+k-n}}{(j+k-n)!}\,
X_{n-k}(t_0)\!\!\left.\vphantom{\sum_n^n}\right]
\nonumber\\
&
=\!\left[\vphantom{\sum_n^n}\right.\!\tilde{X}_n(t)-\sum_{j\in J} P_j(t+t_0) \tilde{X}_j(1-t_0)
\!\!\left.\vphantom{\sum_n^n}\right]
+\sum_{k=0}^{n-1} \left[\vphantom{\sum_n^n}\right.\!\frac{t^k}{k!}
-\sum_{j\in J:j\ge n-k} \frac{(1-t_0)^{j+k-n}}{(j+k-n)!}\, P_j(t+t_0)
\!\!\left.\vphantom{\sum_n^n}\right]
X_{n-k}(t_0).
\label{Y-translated}
\end{align}
We are going to express the $X_{n-k}(t_0)$, $k\in\{1,\dots,n\}$, by means
of the $\tilde{X}_j(1-t_0)$, $j\in J$. We have, by differentiating~(\ref{Y-translated})
$i$ times, for $i\in\{0,1,\dots,n-1\}$,
\begin{align*}
Y^{(i)}(t+t_0)
&
=\!\left[\vphantom{\sum_n^n}\right.\!\tilde{X}_{n-i}(t)-
\sum_{j\in J} P^{(i)}_j(t+t_0) \tilde{X}_j(1-t_0)\!\!\left.\vphantom{\sum_n^n}\right]
\\
&
\hphantom{=\,}
+\sum_{k=0}^{n-1} \left[\vphantom{\sum_n^n}\right.\!\frac{t^{k-i}}{(k-i)!}\ind_{\{k\ge i\}}
-\sum_{j\in J:j\ge n-k} \frac{(1-t_0)^{j+k-n}}{(j+k-n)!}\,
P^{(i)}_j(t+t_0)\!\!\left.\vphantom{\sum_n^n}\right]\!
X_{n-k}(t_0).
\end{align*}
For $t=0$, this yields
\begin{equation}\label{Y-i-tzero}
Y^{(i)}(t_0)=-\sum_{j\in J} P^{(i)}_j(t_0) \tilde{X}_j(1-t_0)
+\sum_{k=0}^{n-1} \left[\vphantom{\sum_n^n}\right.\! \delta_{i,k}
-\sum_{j\in J:j\ge n-k} \frac{(1-t_0)^{j+k-n}}{(j+k-n)!}\,
P^{(i)}_j(t_0)\!\!\left.\vphantom{\sum_n^n}\right]
X_{n-k}(t_0).
\end{equation}
Set for $i,k\in\{0,1,\dots,n-1\}$
$$
a_{ik}=\delta_{i,k}-\sum_{j\in J:j\ge n-k}
\frac{(1-t_0)^{j+k-n}}{(j+k-n)!}\, P^{(i)}_j(t_0)
$$
and let us introduce the matrix
$$
A=(a_{ik})_{0\le i,k\le n-1}
$$
together with its inverse matrix
$$
B=A^{-1}=(b_{ik})_{0\le i,k\le n-1}.
$$
The equalities~(\ref{Y-i-tzero}) for $i\in\{0,1,\dots,n-1\}$ read as a
linear system of $(n-1)$ equations and $(n-1)$ unknowns:
$$
\sum_{k=0}^{n-1} a_{ik} X_{n-k}(t_0)=Y^{(i)}(t_0)
+\sum_{j\in J} P^{(i)}_j(t_0)\tilde{X}_j(1-t_0),\quad i\in\{0,1,\dots,n-1\},
$$
which can be rewritten into a matrix form as
$$
A\begin{pmatrix}
X_n(t_0) \\ X_{n-1}(t_0) \\ \vdots \\ X_2(t_0) \\ X_1(t_0)
\end{pmatrix}
=\begin{pmatrix}
Y(t_0) \\ Y'(t_0) \\ \vdots \\ Y^{(n-2)}(t_0) \\ Y^{(n-1)}(t_0)
\end{pmatrix}
+\sum_{j\in J} \tilde{X}_j(1-t_0)
\begin{pmatrix}
P_j(t_0) \\ P'_j(t_0) \\ \vdots \\ P^{(n-2)}_j(t_0) \\ P^{(n-1)}_j(t_0)
\end{pmatrix}\!.
$$
The solution is given by
$$
\begin{pmatrix}
X_n(t_0) \\ X_{n-1}(t_0) \\ \vdots \\ X_2(t_0) \\ X_1(t_0)
\end{pmatrix}
=B\begin{pmatrix}
Y(t_0) \\ Y'(t_0) \\ \vdots \\ Y^{(n-2)}(t_0) \\ Y^{(n-1)}(t_0)
\end{pmatrix}
+\sum_{j\in J} \tilde{X}_j(1-t_0)
B\begin{pmatrix}
P_j(t_0) \\ P'_j(t_0) \\ \vdots \\ P^{(n-2)}_j(t_0) \\ P^{(n-1)}_j(t_0)
\end{pmatrix}
$$
and we see that $X_{n-k}(t_0)$, $k\in\{0,1,\dots,n-1\}$, is of the form
\begin{equation}\label{X-tzero}
X_{n-k}(t_0)=\sum_{i=0}^{n-1} b_{ki} Y^{(i)}(t_0)
+\sum_{j\in J} \left[\,\sum_{i=0}^{n-1} b_{ki}P^{(i)}_j(t_0)\right]\tilde{X}_j(1-t_0).
\end{equation}
Therefore, by plugging~(\ref{X-tzero}) into~(\ref{Y-translated}), we obtain
\begin{align*}
Y(t+t_0)
&
=\tilde{X}_n(t)-\sum_{j\in J} P_j(t+t_0) \tilde{X}_j(1-t_0)
\\
&
\hphantom{=\,}+\sum_{k=0}^{n-1} \left[\vphantom{\sum_n^n}\right.\!\frac{t^{k}}{k!}
-\sum_{j\in J:j\ge n-k} \frac{(1-t_0)^{j+k-n}}{(j+k-n)!}
\, P_j(t+t_0)\!\!\left.\vphantom{\sum_n^n}\right]
\\
&
\hphantom{=\,}\times\!\left[\vphantom{\sum_n^n}\right.\!\sum_{i=0}^{n-1} b_{ki} Y^{(i)}(t_0)
+\sum_{j\in J} \left(\sum_{i=0}^{n-1} b_{ki}P^{(i)}_j(t_0)\right)\tilde{X}_j(1-t_0)
\!\!\left.\vphantom{\sum_n^n}\right]
\\
&
=\tilde{X}_n(t)-\sum_{j\in J} \left[\vphantom{\sum_n^n}\right.\!P_j(t+t_0)-
\sum_{0\le i,k\le n-1} b_{ki} P^{(i)}_j(t_0)
\\
&
\hphantom{=\,} \times\left(\vphantom{\sum_n^n}\right.\!\frac{t^{k}}{k!}
-\sum_{\jmath\in J:\jmath\ge n-k} \frac{(1-t_0)^{\jmath+k-n}}{(\jmath+k-n)!}\, P_{\jmath}(t+t_0)
\!\!\left.\vphantom{\sum_n^n}\right)\!\!\left.\vphantom{\sum_n^n}\right]\!\tilde{X}_j(1-t_0)
\\
&
\hphantom{=\,}
+\sum_{0\le i,k\le n-1} b_{ki}\!\left[\vphantom{\sum_n^n}\right.\!\frac{t^{k}}{k!}
-\sum_{j\in J:j\ge n-k} \frac{(1-t_0)^{j+k-n}}{(j+k-n)!}\, P_j(t+t_0)
\!\!\left.\vphantom{\sum_n^n}\right]\! Y^{(i)}(t_0).
\end{align*}
Finally, $Y(t+t_0)$ can be written as
$$
Y(t+t_0)=\tilde{Y}_{t_0}(t)+\sum_{i=0}^{n-1} Q_{i,t_0}(t) Y^{(i)}(t_0)
$$
where
$$
Q_{i,t_0}(t)=\sum_{k=0}^{n-1} b_{ki}\!\left[\vphantom{\sum_n^n}\right.\!\frac{t^{k}}{k!}
-\sum_{j\in J:j\ge n-k} \frac{(1-t_0)^{j+k-n}}{(j+k-n)!}\, P_j(t+t_0)
\!\!\left.\vphantom{\sum_n^n}\right]
$$
and
$$
\tilde{Y}_{t_0}(t)=\tilde{X}_n(t)-\sum_{j\in J}\tilde{P}_{j,t_0}(t)\tilde{X}_j(1-t_0)
$$
with
\begin{align*}
\tilde{P}_{j,t_0}(t)
&
=P_j(t+t_0)-\sum_{0\le i,k\le n-1} b_{ki} P^{(i)}_j(t_0)
\!\left[\vphantom{\sum_n^n}\right.\!\frac{t^{k}}{k!}
-\sum_{\jmath\in J:\jmath\ge n-k} \frac{(1-t_0)^{\jmath+k-n}}{(\jmath+k-n)!}
\, P_{\jmath}(t+t_0)\!\!\left.\vphantom{\sum_n^n}\right]
\\
&
=P_j(t+t_0)-\sum_{i=0}^{n-1} P^{(i)}_j(t_0)Q_{i,t_0}(t).
\end{align*}

\noindent$\bullet$ \textsl{Second step.}
We easily see that the functions $\tilde{P}_{j,t_0}$ and $Q_{i,t_0}$ are
polynomials of degree less than $2n$. Let us compute now their derivatives at $0$
and $t_0$. First, concerning $Q_{i,t_0}$ we have
$$
Q^{(\iota)}_{i,t_0}(t)=\sum_{k=0}^{n-1} b_{ki}
\!\left[\vphantom{\sum_n^n}\right.\!\frac{t^{k-\iota}}{(k-\iota)!}
\ind_{\{k\ge \iota\}}-\sum_{j\in J:j\ge n-k}
\frac{(1-t_0)^{j+k-n}}{(j+k-n)!}\, P^{(\iota)}_j(t+t_0)
\!\!\left.\vphantom{\sum_n^n}\right]\!.
$$
Choosing $t=0$ and recalling the definition of $a_{\iota k}$ and the fact
that the matrices $(a_{ik})_{0\le i,k\le n}$ and $(b_{ik})_{0\le i,k\le n}$
are inverse, this gives for $\iota\in\{0,1,\dots,n-1\}$,
$$
Q^{(\iota)}_{i,t_0}(0)=\sum_{k=0}^{n-1} b_{ki} \!\left[
\vphantom{\sum_n^n}\right.\!\delta_{\iota, k}
-\sum_{j\in J:j\ge n-k} \frac{(1-t_0)^{j+k-n}}{(j+k-n)!}\,
P^{(\iota)}_j(t_0)\!\!\left.\vphantom{\sum_n^n}\right]
=\sum_{k=0}^{n-1}a_{\iota k} b_{ki}=\delta_{\iota, i}.
$$
Choosing $t=1-t_0$, we have for $\iota\in I$
\begin{equation}\label{Qitzero}
Q^{(\iota)}_{i,t_0}(1-t_0)=\sum_{k=0}^{n-1} b_{ki}\!\left[
\vphantom{\sum_n^n}\right.\!\frac{(1-t_0)^{k-\iota}}{(k-\iota)!}\,
\ind_{\{k\ge \iota\}}-\sum_{j\in J:j\ge n-k} \frac{(1-t_0)^{j+k-n}}{(j+k-n)!}
\, P^{(\iota)}_j(1)\!\!\left.\vphantom{\sum_n^n}\right]\!.
\end{equation}
By Theorem~\ref{th-drift}, we know that $P^{(\iota)}_j(1)=\delta_{j,n-\iota}$
for $\iota\in I$; then
$$
\sum_{j\in J:j\ge n-k} \frac{(1-t_0)^{j+k-n}}{(j+k-n)!}\, P^{(\iota)}_j(1)
=\frac{(1-t_0)^{k-\iota}}{(k-\iota)!}\,\ind_{\{k\ge \iota\}}\ind_{\{\iota\in(n-J)\}}.
$$
Observing that, if $\iota\le k\le n-1$, the conditions $\iota\in(n-J)$ and
$\iota\in I$ are equivalent, we simply have
$$
\sum_{j\in J:j\ge n-k} \frac{(1-t_0)^{j+k-n}}{(j+k-n)!}\, P^{(\iota)}_j(1)
=\frac{(1-t_0)^{k-\iota}}{(k-\iota)!}\,\ind_{\{k\ge \iota\}}
$$
which immediately entails, by~(\ref{Qitzero}),
$$
Q^{(\iota)}_{i,t_0}(1-t_0)=0\quad\mbox{for }\iota\in I.
$$
Next, concerning $\tilde{P}^{(\iota)}_{j,t_0}$, we have
$$
\tilde{P}^{(\iota)}_{j,t_0}(t)
=P^{(\iota)}_j(t+t_0)-\sum_{i=0}^{n-1} P^{(i)}_j(t_0)Q^{(\iota)}_{i,t_0}(t).
$$
Choosing $t=0$, this gives for $\iota\in\{0,1,\dots,n-1\}$,
since $Q^{(\iota)}_{i,t_0}(0)=\delta_{\iota, i}$,
$$
\tilde{P}^{(\iota)}_{j,t_0}(0)
=P^{(\iota)}_j(t_0)-\sum_{i=0}^{n-1} P^{(i)}_j(t_0)Q^{(\iota)}_{i,t_0}(0)=0.
$$
Choosing $t=1-t_0$, we have for $\iota\in I$, since $P^{(\iota)}_j(1)=\delta_{\iota,n-j}$
and $Q^{(\iota)}_{i,t_0}(1-t_0)=0$,
$$
\tilde{P}^{(\iota)}_{j,t_0}(1-t_0)
=P^{(\iota)}_j(1)-\sum_{i=0}^{n-1} P^{(i)}_j(t_0)Q^{(\iota)}_{i,t_0}(1-t_0)
=\delta_{\iota,n-j}.
$$
The polynomials $\tilde{P}_{j,t_0}$ (resp. $Q_{i,t_0}$) enjoy the same
properties than the $P_j$'s (resp. the $\tilde{\psi}_i$'s),
regarding the successive derivatives, they can be deduced from these
latter by a rescaling according as
$$
\tilde{P}_{j,t_0}(t)=(1-t_0)^{n-j}P_j\!\left(\frac{t}{1-t_0}\right)\quad
(\mbox{resp. }Q_{i,t_0}=(1-t_0)^{i}\tilde{\psi}_i\!\left(\frac{t}{1-t_0}\right)).
$$
It is then easy to extract the identity in distribution below, by using
the property of Gaussian conditioning:
$$
\left(\vphantom{\sum_n^n}\right.\!\tilde{X}_n(t)-\sum_{j\in J}\tilde{P}_{j,t_0}(t)
\tilde{X}_j(1-t_0)\!\!\left.\vphantom{\sum_n^n}\right)_{t\in[0,1-t_0]}
\stackrel{d}{=}
\big(\tilde{X}_n(t) \big| \tilde{X}_j(1-t_0)=0,j\in J\big)_{t\in[0,1-t_0]}.
$$
Theorem~\ref{th-prediction} is established.
\qed
\end{proof}

\section{Example: bridges of integrated Brownian motion ($n=2$)}\label{sect-IBM}

Here, we have a look on the particular case where $n=2$ for which
the corresponding process $(X_n(t))_{t\in [0,1]}$ is nothing but
integrated Brownian motion (the so-called Langevin process):
$$
X_2(t)=\int_0^t B(s)\,\dd s.
$$
The underlying Markov process is the so-called Kolmogorov diffusion
$(X_2(t),X_1(t))_{t\in [0,1]}$.
All the associated conditioned processes that will be constructed are related to the equation
$v^{(4)}(t)=u(t)$ with boundary value conditions at time~$0$:
$v(0)=v'(0)=0$. There are four such processes:
\begin{itemize}
\item
$(X_2(t))_{t\in[0,1]}$ (integrated Brownian motion);
\item
$(X_2(t)|X_1(1)=0)_{t\in[0,1]}$ (integrated Brownian bridge);
\item
$(X_2(t)|X_2(1)=0)_{t\in[0,1]}$ (bridge of integrated Brownian motion);
\item
$(X_2(t)|X_1(1)=X_2(1)=0)_{t\in[0,1]}$ (another bridge of integrated Brownian motion).
\end{itemize}
On the other hand, when adding two boundary value conditions at time~$1$ to the
foregoing equation, we find six boundary value problems:
$v(1)=v'(1)=0$, $v(1)=v''(1)=0$, $v(1)=v'''(1)=0$, $v'(1)=v''(1)=0$,
$v'(1)=v'''(1)=0$, $v''(1)=v'''(1)=0$.
Actually, only four of them can be related to some
Gaussian processes--the above listed processes--in the sense of our work
whereas two others cannot be.

For each process, we provide the covariance function, the representation
by means of integrated Brownian motion subject to a random polynomial drift,
the related boundary value conditions at $1$ and the decomposition related
to the prediction problem.
Since the computations are straightforward, we shall omit them and we only
report here the results.

For an account on integrated Brownian motion in relation with the present work,
we refer the reader to, e.g., \cite{drift} and references therein.

\subsection{Integrated Brownian motion}

The process corresponding to the set $J=\varnothing$ is nothing but integrated Brownian motion:
$$
(X_2(t))_{t\in [0,1]}=\left(\int_0^t B(s)\,\dd s\right)_{t\in [0,1]}.
$$
The covariance function is explicitly given by
$$
c(s,t)=\frac16 [s\wedge t]^2\,[3(s\vee t)-s\wedge t].
$$
This process is related to the boundary value conditions at $1$ ($I=\{2,3\}$):
$v''(1)=v'''(1)=0$.
The prediction property can be stated as follows:
$$
(X_2(t+t_0))_{t\in [0,1-t_0]}=\big(\tilde{X}_2(t)+X_2(t_0)+tX_1(t_0)\big)_{t\in [0,1-t_0]}.
$$

\subsection{Integrated Brownian bridge}

The process corresponding to the set $J=\{1\}$ is integrated Brownian bridge:
$$
(Y(t))_{t\in [0,1]}=\left(\int_0^t B(s)\,\dd s\,\Big|B(1)=0\right)_{t\in [0,1]}
=\left(\int_0^t \beta(s)\,\dd s\right)_{t\in [0,1]}
=(X_2(t)|X_1(1)=0)_{t\in [0,1]}.
$$
This process can be represented as
$$
(Y(t))_{t\in [0,1]}=\left(X_2(t)-\frac12\,t^2 X_1(1)\right)_{t\in [0,1]}.
$$
The covariance function is explicitly given by
$$
c(s,t)=\frac16 [s\wedge t]^2\,[3(s\vee t)-s\wedge t]-\frac14\,s^2t^2.
$$
The process $(Y(t))_{t\in [0,1]}$ is related to the boundary value conditions
at $1$ ($I=\{1,3\}$): $v'(1)=v'''(1)=0$.
The prediction property says that
$$
(Y(t+t_0))_{t\in [0,1-t_0]}=\left(\tilde{Y}_{t_0}(t)+Y(t_0)
+\left(t-\frac{t^2}{2(1-t_0)}\right)Y'(t_0)\right)_{t\in [0,1-t_0]}.
$$

\subsection{Bridge of integrated Brownian motion}

The process corresponding to the set $J=\{2\}$ is the bridge of integrated Brownian motion:
$$
(Y(t))_{t\in [0,1]}=\left(\int_0^t B(s)\,\dd s\,\bigg|\int_0^1 B(s)\,\dd s=0\right)_{t\in [0,1]}
=(X_2(t)|X_2(1)=0)_{t\in [0,1]}.
$$
The bridge is understood as the process is pinned at its extremities: $Y(0)=Y(1)=0$.
This process can be represented as
$$
(Y(t))_{t\in [0,1]}=\left(X_2(t)-\frac12\,t^2(3-t) X_2(1)\right)_{t\in [0,1]}.
$$
The covariance function is explicitly given by
$$
c(s,t)=\frac16 [s\wedge t]^2\,[3(s\vee t)-s\wedge t]-\frac{1}{12}\,s^2t^2(3-s)(3-t).
$$
The process $(Y(t))_{t\in [0,1]}$ is related to the boundary value conditions at $1$ ($I=\{0,2\}$):
$v(1)=v''(1)=0$.
The prediction property says that
\begin{align*}
\lqn{(Y(t+t_0))_{t\in [0,1-t_0]}}
=\left(\tilde{Y}_{t_0}(t)+\frac{t^3-3(1-t_0)t^2+2(1-t_0)^3}{2(1-t_0)^3}\,Y(t_0)
+\frac{t^3-3(1-t_0)t^2+2(1-t_0)^2t}{2(1-t_0)^2}\,Y'(t_0)\right)_{t\in [0,1-t_0]}.
\end{align*}

\subsection{Other bridge of integrated Brownian motion}

The process corresponding to the set $J=\{1,2\}$ is another bridge of integrated Brownian motion
(actually of the two-dimensional Kolmogorov diffusion):
$$
(Y(t))_{t\in [0,1]}=\left(\int_0^t B(s)\,\dd s\,\bigg|\int_0^1 B(s)\,\dd s=B(1)=0\right)_{t\in [0,1]}
=(X_2(t)|X_2(1)=X_1(1)=0)_{t\in [0,1]}.
$$
The bridge here is understood as the process is pinned at its extremities
together with its derivatives: $Y(0)=Y'(0)=Y(1)=Y'(1)=0$.
This process can be represented as
$$
(Y(t))_{t\in [0,1]}=\left(X_2(t)-t^2(t-1) X_1(1)-t^2(3-2t) X_2(1)\right)_{t\in [0,1]}.
$$
The covariance function is explicitly given by
$$
c(s,t)=\frac16 [s\wedge t]^2\,[1-s\vee t]^2\,[3(s\vee t)-s\wedge t-2st].
$$
The process $(Y(t))_{t\in [0,1]}$ is related to the boundary value conditions
at $1$ ($I=\{0,1\}$): $v(1)=v'(1)=0$.
The prediction property says that
$$
(Y(t+t_0))_{t\in [0,1-t_0]}=\left(\tilde{Y}_{t_0}(t)+\frac{t^2(t+t_0-1)}{(1-t_0)^3}\,Y(t_0)
+\frac{t^2(3-3t_0-2t)}{(1-t_0)^2}\,Y'(t_0)\right)_{t\in [0,1-t_0]}.
$$

\subsection{Two counterexamples}

\noindent$\bullet$ The solution of the problem associated with the boundary
value conditions $v(1)=v'''(1)=0$ (which corresponds to the set $I_1=\{0,3\}$)
is given by
$$
v(t)=\int_0^1 G_1(s,t) u(s)\,\dd s, \quad t\in[0,1],
$$
where
$$
G_1(s,t)=\frac16 [s\wedge t]^2\,[3(s\vee t)-s\wedge t]+\frac16\,s^2t^2(s-3).
$$

\noindent$\bullet$ The solution of the problem associated with the boundary
value conditions $v'(1)=v''(1)=0$ (which corresponds to the set $I_2=\{1,2\}$),
is given by
$$
v(t)=\int_0^1 G_2(s,t) u(s)\,\dd s, \quad t\in[0,1],
$$
where
$$
G_2(s,t)=\frac16 [s\wedge t]^2\,[3(s\vee t)-s\wedge t]+\frac16\,s^2t^2(t-3).
$$
We can observe the relationships $G_1(s,t)=G_2(t,s)$ and
$I_2=\{0,1,2,3\}\backslash(3-I_1)$.
The Green functions $G_1$ and $G_2$ are not symmetric, so they cannot be
viewed as the covariance functions of any Gaussian process.
In the next section, we give an explanation of these observations.

\section{General boundary value conditions}\label{sect-general}

In this last part, we address the problem of relating the general boundary
value problem
\begin{equation}\label{BVP-gene-bis}
\begin{cases}
v^{(2n)}=(-1)^nu \quad\mbox{on }[0,1], \\
v(0)=v'(0)=\dots=v^{(n-1)}(0)=0, \\
v^{(i_1)}(1)=v^{(i_2)}(1)=\dots=v^{(i_n)}(1)=0,
\end{cases}
\end{equation}
for any indices $i_1,i_2,\dots,i_n$ such that $0\le i_1<i_2<\dots<i_n\le 2n-1$,
to some possible Gaussian process. Set $I=\{i_1,i_2,\dots,i_n\}$.
We have noticed in Theorem~\ref{th-BVP-gene} and Remark~\ref{remark-unique}
that, when $I$ satisfies the relationship $2n-1-I=\{0,1,\dots,2n-1\}\backslash I$,
the system~(\ref{BVP-gene-bis}) admits a unique solution. We proved
this fact by computing an energy integral. Actually, this fact
holds for any set of indices~$I$; see Lemma~\ref{hermite}.

Our aim is to characterize the set of indices $I$ for which
the Green function of~(\ref{BVP-gene-bis}) can be viewed as
the covariance function of a Gaussian process.
A necessary condition for a function of two variables to be the covariance
function of a Gaussian process is that it must be symmetric.
So, we shall characterize the set of indices $I$ for which
the Green function of~(\ref{BVP-gene-bis}) is symmetric and we shall see
that in this case this function is a covariance function.

\subsection{Representation of the solution}

We first write out a representation for the Green function of~(\ref{BVP-gene-bis}).
%
\begin{theorem}\label{th-BVP-gene-bis}
The boundary value problem~(\ref{BVP-gene-bis}) has a unique solution.
The corresponding Green function admits the following representation,
for $s,t\in[0,1]$:
$$
G_{_{\!I}}(s,t)=(-1)^n\ind_{\{s\le t\}}\frac{(t-s)^{2n-1}}{(2n-1)!}
-(-1)^n\sum_{\iota\in I}\frac{(1-s)^{2n-1-\iota}}{(2n-1-\iota)!}\,
R_{{\scriptscriptstyle I},\iota}(t)
$$
where the $R_{{\scriptscriptstyle I},\iota}$, $\iota\in I$,
are Hermite interpolation polynomials satisfying
\begin{equation}\label{condition-hermite-bis}
\begin{cases}
R_{{\scriptscriptstyle I},\iota}^{(2n)}=0,
\\[.5ex]
R_{{\scriptscriptstyle I},\iota}^{(i)}(0)=0 &\mbox{for } i\in\{0,1,\dots,n-1\},
\\[.5ex]
R_{{\scriptscriptstyle I},\iota}^{(i)}(1)=\delta_{\iota,i}
&\mbox{for } i\in I.
\end{cases}
\end{equation}
\end{theorem}
%
\begin{remark}
The conditions~(\ref{condition-hermite-bis}) characterize the
polynomials $R_{{\scriptscriptstyle I},\iota}$, $\iota\in I$.
We prove this fact in Lemma~\ref{hermite} in the appendix.
\end{remark}
%

\begin{proof}
Let us introduce the functions $v_1$ and $v_2$ defined, for any $t\in[0,1]$, by
$$
v_1(t)=(-1)^n\int_0^t \frac{(t-s)^{2n-1}}{(2n-1)!}\,u(s)\,\dd s\quad
\mbox{and}\quad v_2(t)=v(t)-v_1(t).
$$
We plainly have $v_1^{(2n)}=(-1)^nu$ and
$v_1(0)=v_1'(0)=\dots=v_1^{(n-1)}(0)=0$. Therefore, the function
$v$ solves the system~(\ref{BVP-gene-bis}) if and only if the function
$v_2$ satisfies
\begin{equation}\label{BVP-v2}
\begin{cases}
v_2^{(2n)}=0\quad\mbox{on }[0,1], &
\\[1ex]
v_2^{(i)}(0)=0 & \mbox{for } i\in\{0,1,\dots,n-1\},
\\
\displaystyle v_2^{(i)}(1)=(-1)^{n-1}\int_0^1 \frac{(1-s)^{2n-1-i}}{(2n-1-i)!}
\,u(s)\,\dd s  & \mbox{for } i\in I.
\end{cases}
\end{equation}
Referring to Lemma~\ref{hermite}, the conditions~(\ref{BVP-v2}) mean
that $v_2$ is a Hermite interpolation polynomial
which can be expressed as a linear combination of the $R_{{\scriptscriptstyle I},\iota}$,
$\iota\in I$, defined in Theorem~\ref{th-BVP-gene-bis} as follows:
$$
v_2(t)=\sum_{\iota\in I} v_2^{(i)}(1) R_{{\scriptscriptstyle I},\iota}(t)
=(-1)^{n-1} \int_0^1 \left[\,\sum_{\iota\in I} \frac{(1-s)^{2n-1-\iota}}{(2n-1-\iota)!}\,
R_{{\scriptscriptstyle I},\iota}(t)\right] \!u(s)\,\dd s.
$$
Consequently, the boundary value problem~(\ref{BVP-gene-bis}) admits a unique solution
$v$ which writes
$$
v(t)=v_1(t)+v_2(t)=\int_0^1 G_{_{\!I}}(s,t)u(s)\,\dd s,\quad t\in[0,1],
$$
where $G_{_{\!I}}(s,t)$ is defined in Theorem~\ref{th-BVP-gene-bis}.
The proof is finished.
\qed
\end{proof}

We now state two intermediate results which will be used in the proof of
Theorem~\ref{Green-sym}.
%
\begin{proposition}\label{different}
Let $I_1$ and $I_2$ be two subsets of $\{0,1,\dots,2n-1\}$ with cardinality $n$.
If the sets $I_1$ and $I_2$ are different, then the corresponding Green functions
$G_{_{\!I_1}}$ and $G_{_{\!I_2}}$ are different.
\end{proposition}
%
\begin{proof}
Suppose that $I_1\neq I_2$. If we had $G_{_{\!I_1}}=G_{_{\!I_2}}$,
for any continuous function $u$, the function $v$ defined on $[0,1]$ by
$$
v(t)=\int_0^1 G_{_{\!I_1}}(s,t)u(s)\,\mathrm{d}s=\int_0^1 G_{_{\!I_2}}(s,t)u(s)\,\mathrm{d}s
$$
would be a solution of
the equation $v^{(2n)}=(-1)^nu$ satisfying $v^{(i)}(0)=0$ for $i\in\{0,1,\dots,n-1\}$
and $v^{(i)}(1)=0$ for $i\in I_1\cup I_2$. Since $I_1\neq I_2$ and since $I_1,I_2$
have the same cardinality, there exists an index $i_0$ which belongs
to $I_2\backslash I_1$. For this $i_0$, we would have
$(\partial^{i_0}G_{_{\!I_1}}/\partial t^{i_0})(s,1)=0$ for all $s\in(0,1)$,
or equivalently,
$$
\sum_{\iota\in I_1} \frac{(1-s)^{2n-1-\iota}}{(2n-1-\iota)!}\,
R_{{\scriptscriptstyle I_1},\iota}^{(i_0)}(1)=\frac{(1-s)^{2n-1-i_0}}{(2n-1-i_0)!}.
$$
This is impossible since the exponent $(2n-1-i_0)$ does not appear in
the polynomial on the left-hand side of the foregoing equality.
As a result, $G_{_{\!I_1}}\neq G_{_{\!I_2}}$.
\qed
\end{proof}
%
\begin{proposition}\label{sym}
Let $I_1$ and $I_2$ be two subsets of $\{0,1,\dots,2n-1\}$ with cardinality $n$.
The relationship $G_{_{\!I_1}}(s,t)=G_{_{\!I_2}}(t,s)$ holds for any $s,t\in[0,1]$
(in other words, the integral operators with kernels $G_{_{\!I_1}}$ and $G_{_{\!I_2}}$ are dual)
if and only if the sets $I_1$ and $I_2$ are linked by
$I_2=\{0,1,\dots,2n-1\}\backslash (2n-1-I_1)$.
\end{proposition}
%
\begin{proof}
We have, for any $s,t\in[0,1]$,
$$
(-1)^n[G_{_{\!I_1}}(s,t)-G_{_{\!I_2}}(t,s)]=\frac{(t-s)^{2n-1}}{(2n-1)!}
-\sum_{\iota\in I_1}\frac{(1-s)^{2n-1-\iota}}{(2n-1-\iota)!}
R_{{\scriptscriptstyle I_1},\iota}(t)
+\sum_{\iota\in I_2}\frac{(1-t)^{2n-1-\iota}}{(2n-1-\iota)!}
R_{{\scriptscriptstyle I_2},\iota}(s).
$$
Set, for any $s,t\in[0,1]$,
$$
S(s,t)=(-1)^n[G_{_{\!I_1}}(s,t)-G_{_{\!I_2}}(t,s)].
$$
The polynomial $S$ has a degree less than $2n$ with respect to each variable $s$ and $t$
and satisfy
\begin{align*}
\frac{\partial^{i_1+i_2}S}{\partial s^{i_2}\partial t^{i_1}}\,(s,t)
&
=(-1)^{i_2}\ind_{\{i_1+i_2\le 2n-1\}}\,\frac{(t-s)^{2n-1-i_1-i_2}}{(2n-1-i_1-i_2)!}
\\
&\hphantom{=\;}
-(-1)^{i_2}\sum_{\iota\in I_1}\ind_{\{i_2\le 2n-1-\iota\}}
\,\frac{(1-s)^{2n-1-\iota-i_2}}{(2n-1-\iota-i_2)!}
R_{{\scriptscriptstyle I_1},\iota}^{(i_1)}(t)
\\
&\hphantom{=\;}
+(-1)^{i_1}\sum_{\iota\in I_2}\ind_{\{i_1\le 2n-1-\iota\}}
\,\frac{(1-t)^{2n-1-\iota-i_1}}{(2n-1-\iota-i_1)!}
R_{{\scriptscriptstyle I_2},\iota}^{(i_2)}(s).
\end{align*}
In particular,
\begin{itemize}
\item
for $t=1$ and $i_1\in I_1$, $i_2\in \{0,1,\dots,2n-1\}$,
\begin{align}
\frac{\partial^{i_1+i_2}S}{\partial s^{i_2}\partial t^{i_1}}\,(s,1)
&
=(-1)^{i_2}\ind_{\{i_1+i_2\le 2n-1\}}\,\frac{(1-s)^{2n-1-i_1-i_2}}{(2n-1-i_1-i_2)!}
\nonumber\\
&\hphantom{=\;}
-(-1)^{i_2}\sum_{\iota\in I_1} \ind_{\{i_2\le 2n-1-\iota\}}
\,\frac{(1-s)^{2n-1-\iota-i_2}}{(2n-1-\iota-i_2)!}\,\delta_{\iota,i_1}
\nonumber\\
&\hphantom{=\;}
+(-1)^{i_1}\sum_{\iota\in I_2} \delta_{\iota,2n-1-i_1}\;R_{{\scriptscriptstyle I_2},\iota}^{(i_2)}(s)
\nonumber\\
&=(-1)^{i_1}\ind_{\{i_1\in (2n-1-I_2)\}} R_{{\scriptscriptstyle I_2},\iota}^{(i_2)}(s);
\label{partial1}
\end{align}

\item
for $(s,t)=(1,0)$ and $i_1\in \{0,1,\dots,n-1\}$, $i_2\in I_2$,
\begin{align}
\frac{\partial^{i_1+i_2}S}{\partial s^{i_2}\partial t^{i_1}}\,(1,0)
=(-1)^{i_1+1}\,\frac{\ind_{\{i_1+i_2\le 2n-1\}}}{(2n-1-i_1-i_2)!}
+(-1)^{i_1}\sum_{\iota\in I_2}\frac{\ind_{\{i_1\le 2n-1-\iota\}}}
{(2n-1-\iota-i_1)!}\,\delta_{\iota,i_2}
=0;
\label{partial2}
\end{align}

\item
for $(s,t)=(0,0)$ and $i_1,i_2\in \{0,1,\dots,n-1\}$,
\begin{align}
\frac{\partial^{i_1+i_2}S}{\partial s^{i_2}\partial t^{i_1}}\,(0,0)
=(-1)^{i_2}\,\delta_{i_1+i_2, 2n-1}=0.
\label{partial3}
\end{align}
\end{itemize}

Now we are able to establish the statement of Proposition~\ref{sym}.
\begin{itemize}
\item
If $I_2\neq\{0,1,\dots,2n-1\}\backslash (2n-1-I_1)$, there exists
$i_1\in I_1$ such that $i_1\in (2n-1-I_2)$ and then, in view of~(\ref{partial1}),
$$
\frac{\partial^{i_1}\!S}{\partial t^{i_1}}\,(s,1)\neq 0.
$$
The polynomial $S$ cannot be null, that is, there exist $s,t\in[0,1]$
such that $G_{_{\!I_1}}(s,t)\neq G_{_{\!I_2}}(t,s)$.
\item
If $I_2=\{0,1,\dots,2n-1\}\backslash (2n-1-I_1)$, for any $i_1\in I_1$,
we have $i_1\notin (2n-1-I_2)$ and then, in view of~(\ref{partial1}),
$$
\frac{\partial^{i_1}\!S}{\partial t^{i_1}}\,(s,1)=0\quad
\mbox{for } i_1\in I_1.
$$
Put, for any $i_1\in \{0,1,\dots,n-1\}$,
$\displaystyle\tilde{S}_{i_1}(s)=\frac{\partial^{i_1}\!S}{\partial t^{i_1}}\,(s,0)$.
The polynomial $\tilde{S}_{i_1}$ has a degree less than $2n$.
By~(\ref{partial2}) and~(\ref{partial3}), we have
$$
\begin{cases}
\tilde{S}_{i_1}^{(i_2)}(0)=0 &\mbox{for }i_2\in \{0,1,\dots,n-1\},
\\[1ex]
\tilde{S}_{i_1}^{(i_2)}(1)=0 &\mbox{for } i_2\in I_2,
\end{cases}
$$
from which we deduce, invoking Lemma~\ref{hermite}, that all the polynomials $\tilde{S}_{i_1}$,
$i_1\in \{0,1,\dots,n-1\}$, are null. Finally, the polynomial $S$
has a degree less than to $2n$ with respect to $t$ and satisfies
$$
\begin{cases}
\displaystyle\frac{\partial^{i_1}\!S}{\partial t^{i_1}}\,(s,0)=0
&\mbox{for } i_1\in \{0,1,\dots,n-1\},
\\[1.5ex]
\displaystyle\frac{\partial^{i_1}\!S}{\partial t^{i_1}}\,(s,1)=0
&\mbox{for } i_1\in I_1.
\end{cases}
$$
We can assert, by Lemma~\ref{hermite}, that $S$ is the null-polynomial.
\end{itemize}
The proof of Proposition~\ref{sym} is finished.
\qed
\end{proof}

A necessary condition for $G_{_{\!I}}$ to be the covariance function of a Gaussian process is
that it must be symmetric: $G_{_{\!I}}(s,t)=G_{_{\!I}}(t,s)$ for any $s,t\in[0,1]$.
The theorem below asserts that if the set of indices $I$ is not of the form
displayed in the preamble of Section~\ref{sect-bridges}, that is
$I\neq \{0,1,\dots,\linebreak 2n-1\}\backslash (2n-1-I)$, the Green
function $G_{_{\!I}}$ is not symmetric and consequently
this function can not be viewed as a covariance function, that is we can not
relate the boundary value problem~(\ref{BVP-gene-bis}) to any Gaussian process.

%
\begin{theorem}\label{Green-sym}
The Green function $G_{_{\!I}}$ is symmetric (and it corresponds to a covariance function)
if and only if the set of indices
$I$ satisfies $2n-1-I=\{0,1,\dots,2n-1\}\backslash I$.
\end{theorem}
%
\begin{proof}
Set $I'=\{0,1,\dots,2n-1\}\backslash (2n-1-I)$. By Proposition~\ref{sym},
we see that $G_{_{\!I}}$ is symmetric if and only if
$G_{\raisebox{0ex}[1.45ex]{$\scriptscriptstyle{\!I'}$}}(s,t)=G_{_{\!I}}(s,t)$
for any $s,t\in[0,1]$, that is, by Proposition~\ref{different}, if and only if $I=I'$.
\qed
\end{proof}

We made several verifications with the aid of Maple.
Below is the program we wrote for this.

\begin{footnotesize}
\begin{verbatim}
[> Green_function:=proc(n,setI)    local V,M,S,T,P,setIcomp;
   V:=i->vector(n,[seq(binomial(k,i),k=n..2*n-1)]);
   M:=liste->stackmatrix(seq(V(i),i=liste));
   S:=(s,liste)->Matrix(1,n,[seq((1-s)^(2*n-1-i)/(i!*(2*n-1-i)!),i=liste)]);
   T:=t->Matrix(n,1,[seq([t^k],k=n..2*n-1)]);
   P:=(s,t,liste)->multiply(S(s,liste),inverse(transpose(M(liste))),T(t))[1,1];
   setIcomp:=[op({seq(i,i=0..2*n-1)} minus {seq(2*n-1-i,i=setI)})];
   print(`value of n`=n,`differentiating indices set I_1`=setI);
   print(`Green function for s<t: GI_1(s,t)`
         =sort(simplify((t-s)^(2*n-1)/((2*n-1)!)-P(s,t,setI)),[s,t],plex));
   print(`Green function for s>t: GI_1(s,t)`
         =sort(simplify(-P(t,s,setI)),[s,t],plex));
   print(`symmetry test: GI_1(s,t)=GI_1(t,s)?`
         =evalb(simplify((t-s)^(2*n-1)/((2*n-1)!)-P(s,t,setI)+P(t,s,setI))=0));
   print(`complementary set of 2n+1-I_1: I_2`=setIcomp);
   print(`difference between the two Green functions for s<t: GI_1(s,t)-GI_2(s,t)`
         =sort(simplify(-P(s,t,setI)+P(s,t,setIcomp)),[s,t],plex));
   print(`difference between the two Green functions for s>t: GI_1(s,t)-GI_2(s,t)`
         =sort(simplify(-P(s,t,setIcomp)+P(s,t,setI)),[s,t],plex));
   print(`equality test between the two Green functions for s<t: GI_1(s,t)=GI_2(t,s)?`
         =evalb(simplify((t-s)^(2*n-1)/((2*n-1)!)-P(s,t,setI)+P(t,s,setIcomp))=0));
   print(`equality test between the two Green functions for s>t: GI_1(s,t)=GI_2(t,s)?`
         =evalb(simplify((t-s)^(2*n-1)/((2*n-1)!)-P(s,t,setIcomp)+P(t,s,setI))=0));
   end proc;
\end{verbatim}
\end{footnotesize}
To obtain the two Green functions associated with the sets $I_1$ and
$I_2$, $G_{_{I_1}}$ and $G_{_{I_2}}$, together with the equality test between them,
run the command {\footnotesize \verb+[> Green_function(n,I_1);+}. For instance, the return
of the command {\footnotesize \verb+[> Green_function(5,[2,3,5,6,8]);+} is
\begin{footnotesize}
\begin{verbatim}
   value of n = 3, differentiating indices set I_1 = [1,4,5]
   Green function for s<t: GI_1(s,t) =
         -1/120 s^5 - 1/72 s^4 t^3 + 1/24 s^4 t + 1/18 s^3 t^3 - 1/12 s^3 t^2
   Green function for s>t: GI_1(s,t) =
         -1/120 s^5 + 1/24 s^4 t - 1/72 s^3 t^4 + 1/18 s^3 t^3 - 1/12 s^3 t^2
   symmetry test: GI_1(s,t)=GI_1(t,s)? = false
   complementary set of 2n+1-I_1: I_2 = [2,3,5]
   difference between the two Green functions for s<t:
         GI_1(s,t)-GI_2(s,t) = -1/72 s^4 t^4 + 1/72 s^3 t^4
   difference between the two Green functions for s>t:
         GI_1(s,t)-GI_2(s,t) = 1/72 s^4 t^3 - 1/72 s^3 t^4
   equality test between the two Green functions for s<t: GI_1(s,t)=GI_2(t,s)? = true
   equality test between the two Green functions for s>t: GI_1(s,t)=GI_2(t,s)? = true
\end{verbatim}
\end{footnotesize}

\subsection{Example: bridges of twice integrated Brownian motion ($n=3$)}\label{sect-IBM2}

Here, we have a look on the particular case where $n=3$ for which
the corresponding process $(X_n(t))_{t\in [0,1]}$ is the twice
integrated Brownian motion:
$$
X_3(t)=\int_0^t (t-s)B(s)\,\dd s=\int_0^t \left(\int_0^{s_2} B(s_1)\,\dd s_1\right)\dd s_2.
$$
All the associated conditioned processes that can be constructed are related to the equation
$v^{(6)}(t)=-u(t)$ with boundary value conditions at time~$0$:
$v(0)=v'(0)=v''(0)=0$. There are $2^3=8$ such processes.
Since the computations are tedious and the explicit results are cumbersome,
we only report the correspondance between bridges and boundary value conditions
at time~$1$ through the sets of indices $I$ and $J$. These are written in the table below.
$$
\begin{array}{|@{\hspace{.3em}}c@{\hspace{.3em}}|@{\hspace{.3em}}c@{\hspace{.3em}}|
@{\hspace{.3em}}c@{\hspace{.3em}}|@{\hspace{.3em}}c@{\hspace{.3em}}|
@{\hspace{.3em}}c@{\hspace{.3em}}|@{\hspace{.3em}}c@{\hspace{.3em}}|
@{\hspace{.3em}}c@{\hspace{.3em}}|@{\hspace{.3em}}c@{\hspace{.3em}}|@{\hspace{.3em}}c@{\hspace{.3em}}|}
\hline\vspace{-2ex}&&&&&&&&
\\
\mbox{conditioning set } J & \varnothing & \{1\} & \{2\} & \{3\} & \{1,2\} & \{1,3\} & \{2,3\} & \{1,2,3\}
\\
\vspace{-2ex}&&&&&&&&
\\
\hline\vspace{-2ex}&&&&&&&&
\\
\mbox{differentiating set } I & \{3,4,5\} & \{2,4,5\} & \{1,3,5\} & \{0,3,4\} & \{1,2,5\}
& \{0,2,4\} & \{0,1,3\} & \{0,1,2\}
\\[-2ex]
&&&&&&&&
\\
\hline
\end{array}
$$
The Green functions related to the other sets cannot be related to some Gaussian processes.
The sets are written in the table below with the correspondance $I_2=\{0,1,2,3,4,5\}\backslash (5-I_1)$.
$$
\begin{array}{|@{\hspace{.3em}}c@{\hspace{.3em}}|@{\hspace{.3em}}c@{\hspace{.3em}}|
@{\hspace{.3em}}c@{\hspace{.3em}}|@{\hspace{.3em}}c@{\hspace{.3em}}|
@{\hspace{.3em}}c@{\hspace{.3em}}|@{\hspace{.3em}}c@{\hspace{.3em}}|
@{\hspace{.3em}}c@{\hspace{.3em}}|}
\hline\vspace{-2ex}&&&&&&
\\
\mbox{differentiating set } I_1 & \{0,1,4\} & \{0,1,5\} & \{0,2,5\} & \{0,3,5\} & \{0,4,5\} & \{1,4,5\}
\\
\vspace{-2ex}&&&&&&
\\
\hline\vspace{-2ex}&&&&&&
\\
\mbox{differentiating set } I_2 & \{0,2,3\} & \{1,2,3\} & \{1,2,4\} & \{1,3,4\} & \{2,3,4\} & \{2,3,5\}
\\[-2ex]
&&&&&&
\\
\hline
\end{array}
$$
\subsection{Example: bridges of thrice integrated Brownian motion ($n=4$)}\label{sect-IBM3}

For $n=4$, only the $2^4=16$ following differentiating sets can be related to bridges:
\begin{align*}
\{0,1,2,3\},\,\{0,1,2,4\},\,\{0,1,3,5\},\,\{0,1,4,5\},\,\{0,2,3,6\},\,\{0,2,4,6\},\,\{0,3,5,6\},\,\{0,4,5,6\},
\\
\{1,2,3,7\},\,\{1,2,4,7\},\,\{1,3,5,7\},\,\{1,4,5,7\},\,\{2,3,6,7\},\,\{2,4,6,7\},\,\{3,5,6,7\},\,\{4,5,6,7\}.
\end{align*}

\section*{Appendix: Hermite interpolation polynomials}
\renewcommand{\thelemma}{A.\arabic{lemma}}
\renewcommand{\theremark}{A.\arabic{remark}}
\setcounter{equation}{0}\setcounter{remark}{0}
\renewcommand{\theequation}{A.\arabic{equation}}
%
\begin{lemma}\label{hermite}
Let $a_i$, $i\in\{0,1,\dots,n-1\}$, and $b_i$, $i\in I$, be real numbers.
There exists a unique polynomial $P$ such that
\begin{equation}\label{condition-hermite-ter}
\begin{cases}
P^{(2n)}=0,
\\[.5ex]
P^{(i)}(0)=a_i &\mbox{for } i\in\{0,1,\dots,n-1\},
\\[.5ex]
P^{(i)}(1)=b_i
&\mbox{for } i\in I.
\end{cases}
\end{equation}
\end{lemma}
%
\begin{remark}
The conditions~(\ref{condition-hermite-ter}) characterize the Hermite
interpolation polynomial at points~$0$ and~$1$ with given values of the
successive derivatives at~$0$ up to order $n-1$ and given values of
the derivatives at~$1$ with \textsl{selected} orders in $I$.
When $I\neq \{0,1,\dots,n-1\}$, these polynomials differ from the usual Hermite interpolation
polynomials which involve the successive derivatives at certain points
\textsl{progressively} from zero order up to certain orders.
\end{remark}
%
\begin{proof}
We look for polynomials $P$ in the form
$P(t)=\sum_{j=0}^{2n-1} c_j\,\frac{t^j}{j!}.$
We have
$$
P^{(i)}(t)=\sum_{j=0}^{2n-1} c_j\,\frac{t^{j-i}}{(j-i)!}.
$$
We shall adopt the convention $1/[(j-i)!]=0$ for $i>j$.
The conditions~(\ref{condition-hermite-ter}) yield the linear system
(with the convention that $i$ and $j$ denote respectively the raw and
column indices)
$$
\begin{cases}
c_i=a_i &\mbox{if } i\in\{0,1,\dots,n-1\},
\\
\displaystyle\sum_{j=0}^{2n-1} \frac{c_j}{(j-i)!}=b_i &\mbox{if } i\in I.
\end{cases}
$$
The $(2n)\times(2n)$ matrix of this system writes
$$
\mathbf{A}=\begin{pmatrix}
(\delta_{i,j})_{\hspace{-.2em}i\in\{0,1,\dots,n-1\}\atop j\in\{0,1,\dots,2n-1\}}
\\[-.5ex]
\dotfill
\\
\begin{pmatrix}\displaystyle\frac{1}{(j-i)!}
\end{pmatrix}_{\!\!\hspace{-4em}i\in I\atop \!\!j\in\{0,1,\dots,2n-1\}}
\end{pmatrix}=\begin{pmatrix}
(\delta_{i,j})_{i\in\{0,1,\dots,n-1\}\atop j\in\{0,1,\dots,n-1\}}
&&
\;(0)_{\hspace{-1.4em}i\in\{0,1,\dots,n-1\}\atop j\in\{n,n+1,\dots,2n-1\}}
\\[-0.5ex]
\dotfill &&\hspace*{-.74em} \dotfill
\\
\begin{pmatrix}\displaystyle\frac{1}{(j-i)!}
\end{pmatrix}_{\!\!\hspace{-3.8em}i\in I\atop \!\!j\in\{0,1,\dots,n-1\}}
&&
\;\displaystyle\left(\frac{1}{(j-i)!}\right)_{\!\!\hspace{-5.1em}i\in I\atop \!\!j\in\{n,n+1,\dots,2n-1\}}
\end{pmatrix}\!.
\hspace{-13.23em}\begin{array}{c}\\[-2.86ex]\vdots\\[-1.5ex]\vdots\\[-1.5ex]\vdots\\[-2.5ex]\vdots\\[-1.5ex]\vdots\end{array}\hspace{13em}
$$
Proving the statement of Lemma~\ref{hermite} is equivalent to proving that
the matrix $\mathbf{A}$ is regular. In view of the form of $\mathbf{A}$
as a bloc-matrix, we see, since the north-west bloc is the unit matrix
and the north-east bloc is the null matrix,
that this is equivalent to proving that the south-east bloc of $\mathbf{A}$
is regular. Let us call this latter $\mathbf{A}_0$ and label its columns
$C_j^{(0)}$, $j\in\{0,1,\dots,n-1\}$:
$$
\mathbf{A}_0=\begin{pmatrix} \displaystyle\frac{1}{(n+j-i)!}
\end{pmatrix}_{\!\!\hspace{-3.7em}i\in I\atop \!\!j\in\{0,1,\dots,n-1\}}
=\begin{pmatrix} \mathbf{C}_0^{(0)} & \mathbf{C}_1^{(0)} & \cdots & \mathbf{C}_{n-1}^{(0)}
\end{pmatrix}\!.
$$
For proving that $\mathbf{A}_0$ is regular, we factorize $\mathbf{A}_0$
into the product of two regular triangular matrices.
The method consists in performing several transformations on the columns of
$\mathbf{A}_0$ which do not affect its rank. We provide in this way an algorithm
leading to a $\mathbf{L}\mathbf{U}$-factorization of $\mathbf{A}_0$
where $\mathbf{L}$ is a lower triangular matrix and $\mathbf{U}$ is an upper triangular matrix
with no vanishing diagonal term.

We begin by performing the transformation
$$
\mathbf{C}_j^{(1)}=\begin{cases}
\mathbf{C}_j^{(0)} &\mbox{if } j=0,
\\[1ex]
\displaystyle \mathbf{C}_j^{(0)}-\frac{\mathbf{C}_{j-1}^{(0)}}{n+j-i_1}
&\mbox{if } j\in\{1,2,\dots,n-1\}.
\end{cases}
$$
The generic term of the column $\mathbf{C}_j^{(1)}$, for $j\in\{1,2,\dots,n-1\}$, is
$$
\frac{1}{(n+j-i)!}-\frac{1}{n+j-i_1}\,\frac{1}{(n+j-i-1)!}
=\frac{i-i_1}{n+j-i_1}\,\frac{1}{(n+j-i)!}.
$$
This transformation supplies a matrix $\mathbf{A}_1$ with columns $\mathbf{C}_j^{(1)}$,
$j\in\{0,1,\dots,n-1\}$, which writes
$$
\mathbf{A}_1=\begin{pmatrix}
\mathbf{C}_0^{(1)} & \mathbf{C}_1^{(1)} & \cdots & \mathbf{C}_{n-1}^{(1)}
\end{pmatrix}
=\begin{pmatrix}\\[-1.8ex]\begin{pmatrix}
\displaystyle\frac{1}{(n-i)!}\end{pmatrix}_{\!\scriptscriptstyle i\in I}\;\;
&
\begin{pmatrix}
\displaystyle\frac{i-i_1}{n+j-i_1}\,\frac{1}{(n+j-i)!}\end{pmatrix}_{\!\!\hspace{-3.7em}i\in I\atop \!\!j\in\{1,2,\dots,n-1\}}
\end{pmatrix}\!.
\hspace{-18.1em}\begin{array}{c}\\[-3.5ex]\vdots\\[-1.5ex]\vdots\end{array}\hspace{14em}
$$
We have written
$$
\mathbf{A}_1=\mathbf{A}_0\mathbf{U_1}
$$
where $\mathbf{U}_1$ is the triangular matrix with a diagonal made of~$1$ below:
$$
\mathbf{U}_1=\left(\delta_{i,j}-\frac{\delta_{i,j-1}\ind_{\{j\ge 1\}}}{n+j-i_1}\right)
_{\!0\le i,j\le n-1}\!.
$$
We now perform the transformation
$$
\mathbf{C}_j^{(2)}=\begin{cases}
\mathbf{C}_j^{(1)} &\mbox{if } j\in\{0,1\},
\\[1ex]
\displaystyle \mathbf{C}_j^{(1)}-\frac{n+j-i_1-1}{n+j-i_1}
\,\frac{\mathbf{C}_{j-1}^{(1)}}{n+j-i_2}
&\mbox{if } j\in\{2,3,\dots,n-1\}.
\end{cases}
$$
The generic term of the column $\mathbf{C}_j^{(2)}$, for $j\in\{2,3,\dots,n-1\}$, is
\begin{align*}
\lqn{\frac{i-i_1}{n+j-i_1}\,\frac{1}{(n+j-i)!}
-\frac{i-i_1}{(n+j-i_1)(n+j-i_2)}\,\frac{1}{(n+j-i-1)!}}
&
=\frac{(i-i_1)(i-i_2)}{(n+j-i_1)(n+j-i_2)}\,\frac{1}{(n+j-i)!}.
\end{align*}
This transformation supplies a matrix $\mathbf{A}_2$ with columns $\mathbf{C}_j^{(2)}$,
$j\in\{0,1,\dots,n-1\}$, which writes
\begin{align*}
\mathbf{A}_2
&
=\begin{pmatrix}
\mathbf{C}_0^{(2)} & \mathbf{C}_1^{(2)} & \cdots & \mathbf{C}_{n-1}^{(2)}
\end{pmatrix}
=\left(
\vphantom{\begin{pmatrix}\displaystyle\frac{(i-i_1)(i-i_2)}{(n+j-i_1)(n+j-i_2)}\,\frac{1}{(n+j-i)!}
\end{pmatrix}_{\!\!\hspace{-3.7em}i\in I\atop \!\!j\in\{2,3,\dots,n-1\}}}
\begin{pmatrix}
\displaystyle\frac{1}{(n-i)!}\end{pmatrix}_{\!\scriptscriptstyle i\in I}\;\;
\begin{pmatrix}
\displaystyle\frac{i-i_1}{n+1-i_1}\,\frac{1}{(n+1-i)!}\end{pmatrix}_{\!\scriptscriptstyle i\in I}\;\;
\right.
\hspace{-13.9em}\begin{array}{c}\\[-3.5ex]\vdots\\[-1.5ex]\vdots\end{array}\hspace{32em}
\\
&
\hspace{17.5em}\left.\begin{pmatrix}
\displaystyle\frac{(i-i_1)(i-i_2)}{(n+j-i_1)(n+j-i_2)}\,\frac{1}{(n+j-i)!}
\end{pmatrix}_{\!\!\hspace{-3.7em}i\in I\atop \!\!j\in\{2,3,\dots,n-1\}}
\right)\!.
\hspace{-23.9em}\begin{array}{c}\\[-3.5ex]\vdots\\[-1.5ex]\vdots\end{array}
\end{align*}
We have written
$$
\mathbf{A}_2=\mathbf{A}_1\mathbf{U}_2=\mathbf{A}_0\mathbf{U}_1\mathbf{U}_2
$$
where $\mathbf{U}_2$ is the triangular matrix with a diagonal made of~$1$ below:
$$
\mathbf{U}_2=\left(\delta_{i,j}-\frac{n+j-i_1-1}{(n+j-i_1)(n+j-i_2)}
\,\delta_{i,j-1}\ind_{\{j\ge 2\}}\right)_{\!0\le i,j\le n-1}\!.
$$
In a recursive manner, we easily see that we can construct a sequence of matrices $\mathbf{A}_k,
\mathbf{U}_k$, $k\in\{1,2,\dots,n-1\}$, such that $\mathbf{A}_k=\mathbf{A}_{k-1}\mathbf{U}_k$
where $\mathbf{U}_k$ is the triangular matrix with a diagonal made of~$1$ below:
$$
\mathbf{U}_k=\left(\delta_{i,j}-\frac{(n+j-i_1-1)\dots(n+j-i_{k-1}-1)}{(n+j-i_1)
\dots(n+j-i_k)}\,\delta_{i,j-1}\ind_{\{j\ge k\}}\right)_{\!0\le i,j\le n-1}
$$
and
\begin{align*}
\lqn{\mathbf{A}_k
=\left(\begin{matrix}\\[-1.8ex]
\!\vphantom{\begin{pmatrix}\frac{1}{(n-i)!}\end{pmatrix}_{i\in I\atop j\in\{k+1,\dots,n-1\}}}
\begin{pmatrix}
\displaystyle\frac{1}{(n-i)!}\end{pmatrix}_{\!\scriptscriptstyle i\in I}\;\;
&
\begin{pmatrix}
\displaystyle\frac{i-i_1}{n+1-i_1}\,\frac{1}{(n+1-i)!}\end{pmatrix}_{\!\scriptscriptstyle i\in I}\;\;
&
\cdots
\end{matrix}\right.
\hspace{-15.4em}\begin{array}{c}\\[-3.5ex]\vdots\\[-1.5ex]\vdots\end{array}
\hspace{12.5em}\begin{array}{c}\\[-3.5ex]\vdots\\[-1.5ex]\vdots\end{array}\hspace{19.4em}
\hspace{-17.7em}\begin{array}{c}\\[-3.5ex]\vdots\\[-1.5ex]\vdots\end{array}
}
&
\hspace{-4.8em}
\begin{pmatrix}
\displaystyle\frac{(i-i_1)\dots(i-i_{k-1})}{(n+k-1-i_1)\dots(n+k-1-i_{k-1})}
\,\frac{1}{(n+k-1-i)!}\end{pmatrix}_{\!\scriptscriptstyle i\in I}\;\;
\hspace{-27.55em}\begin{array}{c}\\[-3.5ex]\vdots\\[-1.5ex]\vdots\end{array}
\\
&
\hspace{-4.8em}
\left.\begin{matrix}\\[-1.8ex]\begin{pmatrix}
\displaystyle\frac{(i-i_1)\dots(i-i_k)}{(n+j-i_1)\dots(n+j-i_k)}\,\frac{1}{(n+j-i)!}
\end{pmatrix}_{\hspace{-3.4em}i\in I\atop \!\!j\in\{k,\dots,n-1\}}
\end{matrix}\right)\!.
\hspace{-24.9em}\begin{array}{c}\\[-3.5ex]\vdots\\[-1.5ex]\vdots\end{array}
\end{align*}
We finally obtain, since all the $U_k$, $k\in\{1,2,\dots,n-1\}$, are regular, that
$$
\mathbf{A}_0=\mathbf{A}_{n-1}\mathbf{U}_{n-1}^{-1}\dots\mathbf{U}_1^{-1}
=\mathbf{L}\mathbf{U}
$$
with $\mathbf{U}=\mathbf{U}_{n-1}^{-1}\dots\mathbf{U}_1^{-1}$ and
\begin{align*}
\mathbf{L}=\mathbf{A}_{n-1}=
\left(\begin{pmatrix}
\displaystyle\frac{1}{(n-i)!}\end{pmatrix}_{\!\scriptscriptstyle i\in I}\;\;
\begin{pmatrix}
\displaystyle\frac{i-i_1}{n+1-i_1}\,\frac{1}{(n+1-i)!}\end{pmatrix}_{\!\scriptscriptstyle i\in I}\;\;
\,\cdots\right.
\hspace{-15.5em}\begin{array}{c}\\[-3.5ex]\vdots\\[-1.5ex]\vdots\end{array}
\hspace{12.48em}\begin{array}{c}\\[-3.5ex]\vdots\\[-1.5ex]\vdots\end{array}\hspace{19.6em}
\hspace{-17.9em}\begin{array}{c}\\[-3.5ex]\vdots\\[-1.5ex]\vdots\end{array}\hspace{19em}
\\
\left.\;\;\begin{pmatrix}
\displaystyle\frac{(i-i_1)\dots(i-i_{n-1})}{(2n-1-i_1)\dots(2n-1-i_{n-1})}
\,\frac{1}{(2n-1-i)!}\end{pmatrix}_{\!\scriptscriptstyle i\in I}
\right)\!.
\hspace{-24em}\begin{array}{c}\\[-3.5ex]\vdots\\[-1.5ex]\vdots\end{array}\hspace{28.1em}
\end{align*}
It is clear that the matrices $\mathbf{U}$ and $\mathbf{L}$ are triangular and regular,
and then $\mathbf{A}_0$ (and $\mathbf{A}$) is also regular. Moreover, the inverse
of $\mathbf{A}_0$ can be computed as
$$
\mathbf{A}_0^{-1}=\mathbf{U}^{-1}\mathbf{L}^{-1}
=\mathbf{U}_1\dots \mathbf{U}_{n-1}\mathbf{L}^{-1}.
$$
The proof of Lemma~\ref{hermite} is finished.
\qed
\end{proof}


\end{document}